\documentclass[11pt]{article}
\usepackage{amssymb,latexsym,amsmath,amsbsy,amsthm,amsxtra,amsgen,graphicx,color,soul,psfrag,float}
\usepackage{mathrsfs}
\newfont{\msbm}{msbm10 at 11pt}
\newcommand {\R} {\mbox{\msbm R}}

\newcommand {\N} {\mbox{\msbm N}}

\def\Var{\textup{Var}}

\newtheorem{Theo}{Theorem}
\newtheorem{Lemma}[Theo]{Lemma}
\newtheorem{Cor}[Theo]{Corollary}
\newtheorem{Prop}[Theo]{Proposition}
\newtheorem{Exm}[Theo]{Example}

\begin{document}
\title{The evolving beta coalescent}

\author{G\"otz Kersting\thanks{Supported in part by DFG SPP 1590 {\em Probabilistic Structures in Evolution}} , Jason Schweinsberg\thanks{Supported in part by NSF Grant DMS-1206195} \  and Anton Wakolbinger$^*$}
\maketitle

\begin{abstract}
In mathematical population genetics, it is well known that one can represent the genealogy of a population by a tree, which indicates how the ancestral lines of individuals in the population coalesce as they are traced back in time.  As the population evolves over time, the tree that represents the genealogy of the population also changes, leading to a tree-valued stochastic process known as the evolving coalescent.  Here we will consider the evolving coalescent for populations whose genealogy can be described by a beta coalescent, which is known to give the genealogy of populations with very large family sizes.  We show that as the size of the population tends to infinity, the evolution of certain functionals of the beta coalescent, such as the total number of mergers, the total branch length, and the total length of external branches, converges to a stationary stable process.  Our methods also lead to new proofs of known asymptotic results for certain functionals of the non-evolving beta coalescent.
\end{abstract}
\section{Introduction}

In 1999, Pitman \cite{pit99} and Sagitov \cite{sag99} introduced coalescents with multiple mergers, also known as $\Lambda$-coalescents.  These processes are continuous-time Markov processes taking values in the set of partitions of $\N$, and as time goes forward, blocks of the partition merge together.  For any finite measure $\Lambda$ on $[0,1]$, the $\Lambda$-coalescent is defined by the property that whenever the restriction of the process to $\{1, \dots, n\}$ has $b$ blocks, each possible transition that involves $b$ blocks merging into one happens at rate $$\lambda_{b,k} = {\int_0^1 p^{k-2} (1-p)^{b-k} \: \Lambda(dp)},$$ and these are the only transitions that occur.  This means that when there are $b$ blocks, the total rate of all mergers is $$\lambda_b = \sum_{k=2}^b \binom{b}{k} \lambda_{b,k}.$$  When $\Lambda$ is a unit mass at zero, only two blocks ever merge at a time, and each transition that involves two blocks merging into one happens at rate $1$, so we get the celebrated Kingman's coalescent \cite{king82}.  When $\Lambda$ is the uniform distribution on $[0,1]$, the $\Lambda$-coalescent is known as the Bolthausen-Sznitman coalescent \cite{bosz98}.

Coalescent processes arise naturally in population genetics, where they are used to model the genealogy of populations.  The genealogy of a population of size $n$ can be modeled by a process $(\Pi_n(r), r \geq 0)$ taking its values in the set of partitions of $\{1, \dots, n\}$.  The integers $i$ and $j$ are in the same block of the partition $\Pi_n(r)$ if and only if the $i$th and $j$th individuals in the population have the same ancestor $r$ units back in time.  Typically, Kingman's coalescent is used to model the genealogy of populations.  However, in some circumstances, other $\Lambda$-coalescents could describe the genealogy of a population.  For example, it was shown in \cite{sch03} that if the probability of an individual having $k$ or more offspring is proportional to $k^{-\alpha}$, where $1 < \alpha < 2$, then the genealogy of the population is best described by the $\Lambda$-coalescent in which $\Lambda$ is the Beta$(2 - \alpha, \alpha)$ distribution.  This process will hereafter be called the Beta$(2 - \alpha, \alpha)$-coalescent, or simply the beta coalescent.  Beta coalescents are thus natural models for populations with large family sizes, and predictions from beta coalescents were shown in \cite{marine} to fit genetic data from some marine species.

There has been recent interest in describing not only the genealogy of a population at a fixed time but also how the genealogy of a population changes over time.  At any given time, the genealogical structure described by the coalescent process can also be represented by a tree.  The shape of this tree changes over time, and the associated tree-valued process $({\cal T}_n(t), t \in \mathbb R)$ is known as an evolving coalescent.  For populations whose genealogy at a fixed time is described by Kingman's coalescent, the associated evolving coalescent was studied in \cite{pw06, pww09, gpw08}. The evolving Bolthausen-Sznitman coalescent was studied in \cite{sch12}.

In the present paper, we study the evolving coalescent for populations whose genealogy is given by the Beta$(2 - \alpha, \alpha)$-coalescent, where $1 < \alpha < 2$.  To this end we consider a time rescaling depending on $n$, the {\em scaled} time being $s=n^{\alpha-1} t$. 
On the original time scale, merging events occur at a rate of order $n^\alpha$; see formula \eqref{rate}  below. Therefore the number of coalescent events in one unit of the scaled time is of order $n$, meaning that the rate at which a specific one of the $n$ lineages takes part in a merging event is of order 1. Thus, while the original time $t$ captures {\em evolutionary time}, the scaled time $s$ plays the role of a {\em generation time}.

We show that as $n \rightarrow \infty$, the distribution of certain functionals of the beta coalescent converges to a stable distribution of index $\alpha$.  For the evolving beta coalescent, the distribution of these functionals evaluated at times $ s_1 < s_2 <\cdots < s_d$ converges as $n \rightarrow \infty$ to a multivariable stable distribution of index $\alpha$.  Examples of functionals that fit into this framework include the total number of merger events before all of the lineages have coalesced into a single lineage, the total length of all branches in the tree, and the total length of all external branches in the tree.

As a typical result, which follows by combining  Example \ref {Lnex} and Corollary \ref{Cor2} below, we state the following theorem, which shows how the total branch length of the coalescent tree evolves over time.
\begin{Theo}
Let $\mathcal L_n (t)$ be the total branch length of ${\cal T}_n(t)$, $t \in \mathbb R$.  Then for $1 < \alpha < \frac{1}{2}(1 + \sqrt{5})$ the sequence of processes
$$n^{\alpha - 1 - 1/\alpha} \bigg( \mathcal L_n(n^{1-\alpha}s) - \frac{\alpha (\alpha - 1) \Gamma(\alpha) n^{2 - \alpha}}{2 - \alpha} \bigg) , \hspace{.3in}  -\infty < s < \infty, $$
converges in finite-dimensional distributions as $n \to \infty$ to the moving average process
\begin{equation}\label{maprocess}
\int_{0}^{\infty} g(r)\: dL_{s-r} , \hspace{.3in}  -\infty < s < \infty,
\end{equation}
where $g(r) = (\alpha-1)(\alpha\Gamma(\alpha))^{\frac{1}{\alpha-1}}(r+\alpha\Gamma(\alpha))^{-\frac{2-\alpha}{\alpha-1}} $ and $(L_s)_{-\infty < s < \infty} $, is a mean zero L\'evy process with $L_0 =0$ and L\'evy measure 
\begin{equation}\label{Levymeas}
 \frac{1}{\Gamma(\alpha) \Gamma(2 - \alpha)} u^{-1 - \alpha} \: du, \hspace{.3in} u>0.
 \end{equation}
\end{Theo}
The time reversal in the integrator is explained in more detail in Section \ref{secev}. The stochastic integral is well-defined for $\alpha < \frac 12(1+ \sqrt 5)$.  Above this threshold, the statement fails to be true even for the fixed value $s=0$, as shown in \cite{kersting}.  Note that the function $g(r)$ decreases as a power of $r$, and consequently the moving average process defined in (\ref{maprocess}) is not Markovian.  The exponent $-(2 - \alpha)/(\alpha - 1)$ tends to $-\infty$ as $\alpha \rightarrow 1$.  In the limiting case $\alpha = 1$, which corresponds to the Bolthausen-Sznitman coalescent, it is known that the evolution of the total branch length converges to a moving average process defined as in (\ref{maprocess}), but with $g(r) = e^{-r}$.  In this case the limit process is Markovian.  For details, see \cite{sch12}.

Our approach consists of deriving asymptotic expansions for suitable functionals by means of Poisson integrals. Thereby we rediscover some known results for functionals of the static (non-evolving) beta coalescent, e.g. for its total length $\mathcal L_n$ and total external length $\ell_n$. Moreover by means of the Poisson integral representations we get hold also on  their (properly scaled) joint distributions, which are asymptotically multivariate stable. This allows us to calculate, for example, the asymptotic distribution of $\ell_n/\mathcal L_n$ (see Example \ref{lnLn} in Section \ref{Fubecoa}).

We give a precise construction of the evolving beta coalescent, as well as a precise statement of the main results and a few examples in section \ref{resultssec}.  Proofs are given in section~\ref{proofsec}.

\section{ Framework, main results, and examples}\label{resultssec}

\subsection{Construction of the evolving beta coalescent}\label{constsec}

We give here a precise construction of the evolving $\Lambda$-coalescent, which is modeled after the Poisson process construction of the $\Lambda$-coalescent given in \cite{pit99} and is similar to the construction of the evolving Bolthausen-Sznitman coalescent in \cite{sch12}.  Let $\Lambda$ be a finite measure on $(0,1]$.  We will construct a population of fixed size $n$ defined for all times $t \in \R$ whose genealogy is given by the $\Lambda$-coalescent.  Individuals in the population will be labeled $1, \dots, n$.

Let $ \Upsilon=\Upsilon_n$ be a Poisson point process on $\R \times (0,1] \times [0,1]^n$ with intensity measure 
$$dt \times {p^{-2} \Lambda(dp) \times dv_1 \dots dv_n}.$$  Suppose $(t, p, v_1, \dots, v_n)$ is a point of $\Upsilon$.  If zero or one of the points $v_1, \dots, v_n$ is less than $p$, then no change in the population occurs at time $t$.  However, if $k \geq 2$ of these points are less than $p$, so that $v_{i_1} < \dots < v_{i_k} < p$, then at time $t$, the individuals labeled $i_2, \dots, i_k$ all die, and the individual labeled $i_1$ gives birth to $k - 1$ new individuals who take over the labels $i_2, \dots, i_k$.  This implies that if we are following the genealogy of the population backwards in time, the lineages labeled $i_1, \dots, i_k$ will all coalesce at time $t$.  To see that the $\Lambda$-coalescent describes the genealogy of this population, note that the rate of events that cause the lineages $i_1, \dots, i_k$ to coalesce is $$\int_0^1 p^k (1-p)^{n-k} \cdot p^{-2} \Lambda(dp) = \lambda_{n,k}.$$  This construction is well-defined because the rate of changes in the population is bounded above by $$\int_0^1 \binom{n}{2} p^2 \cdot p^{-2} \Lambda(dp) < \infty.$$
Note that, since we are ordering the genealogy with respect to the $v_i$, and not in a lookdown manner with respect to the indices $i$, we do not have strong consistency in $n$.

Although this construction works whenever $\Lambda(\{0\}) = 0$, we will hereafter restrict ourselves to the case in which $\Lambda$ is the Beta$(2 - \alpha, \alpha)$ distribution.  For each $t \in \R$, there will be a different realization of the beta coalescent which describes the genealogy of the population at time $t$.  We denote the corresponding coalescent tree (read off from the genealogy backwards from time  $t$)  by ${\cal T}_n(t)$ and call the process $({\cal T}_n(t), t \in \R)$ the evolving beta coalescent.

\subsection{Two Poisson processes}\label{twoPPP}

The evolving beta coalescent is constructed  from a Poisson process $\Upsilon_n$ on $\R \times (0,1] \times [0,1]^n$ with intensity measure \begin{align*}
dt \times \frac{1}{\Gamma(\alpha) \Gamma(2 - \alpha)} p^{-1-\alpha} (1 - p)^{\alpha - 1} \: dp \times dv_1 \dots dv_n.
\end{align*}
  In this section, we will construct two other Poisson processes, denoted by $\Psi_n$ and $\Theta_n$, that will be useful for analyzing   the evolving  beta coalescent and the static (non-evolving) beta coalescent back from time 0, respectively.

First, we obtain a Poisson process $\Upsilon'$ on $\R \times \R^+$  in two steps by discarding all but the first two coordinates of the points of $\Upsilon_n$, and then augmenting these points with the points of an independent Poisson process with intensity $$dt \times \frac{1}{\Gamma(\alpha) \Gamma(2 - \alpha)} p^{-1-\alpha} (1 - q(p)) \: dp,$$  where $q(p) = 0$ if $p \geq 1$ and $q(p) = (1 - p)^{\alpha - 1}$ if $0 < p < 1$.  Note that 
$\Upsilon'$ has intensity
\begin{equation}\label{psiintensity}
dt \times \frac{1}{\Gamma(\alpha) \Gamma(2 - \alpha)} p^{-1 - \alpha} \: dp.
\end{equation}
From $\Upsilon'$, we can then obtain a Poisson process $\Psi_n$ via the mapping $$(t, p) \mapsto (s,u):= (n^{\alpha-1} t, n^{1-1/\alpha} p).$$ \begin{figure}[H]
\psfrag{y}{$\Upsilon'$}
\psfrag{0}{$t=0$}
\psfrag{2}{$s=0$}
\psfrag{t}{$t$}
\psfrag{u}{$u$}
\psfrag{p}{$p$}
\psfrag{s}{$s$}
\psfrag{1}{$p=1$}
\psfrag{q}{$\Psi_n$}
\includegraphics[width=16cm]{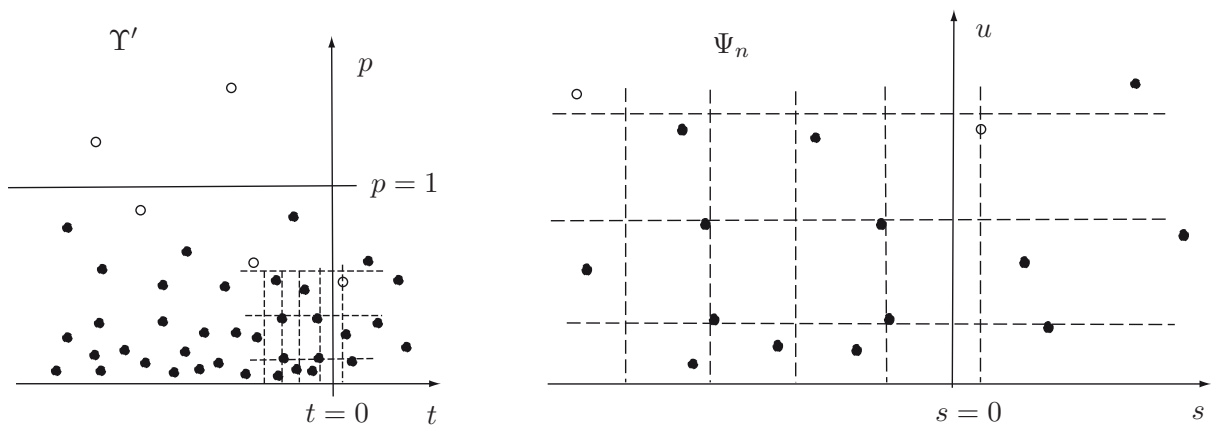}
\caption{The process $\Upsilon'$ contains the points (marked by $\bullet$) from the first two coordinates of the points of $\Upsilon$, plus additional points (marked by $\circ$) that  make up for the difference between the intensities $p^{-1-\alpha} (1-p)^{\alpha-1}  1_{p\le 1}$ and $p^{-1-\alpha}$, $p>0$. The point process $\Psi_n$ arises from $\Upsilon'$ through the transformation  $(t,p) \mapsto (s,u)=(n^{\alpha-1}t, n^{1-1/\alpha}p)$.}
\end{figure}
That is, if $(t,p)$ is a point of $\Upsilon'$, then $(n^{\alpha-1} t, n^{1-1/\alpha} p)$ is a point of $\Psi_n$.  It is straightforward to verify that the intensity of $\Psi_n$ is also given by (\ref{psiintensity}), now with $(t,p)$ replaced by $(s,u)$.

The reason for considering the rescaled Poisson process $\Psi_n$ is that the fluctuations in the behavior of the beta coalescent that will be important for studying the evolving coalescent are those that happen after the coalescent has evolved for a time which is $O(n^{1-\alpha})$.  Also, the largest mergers on this time scale affect approximately a fraction $O(n^{-1 + 1/\alpha})$ of the blocks, in other words a number of $O(n^{1/\alpha})$ blocks out of the $n$ initial blocks.  If $(s, u)$ is a point of $\Psi_n$, then if this point corresponds to a point of $\Upsilon_n$ (i.e. unless it corresponds to a point of $\Upsilon'$ that does not belong to $\Upsilon_n$), at time $n^{1-\alpha}s$ there is an event during which approximately a fraction $n^{-1 + 1/\alpha} u$ of the blocks merge together.

We now define another Poisson process $\Theta_n$ on ${(0,1]} \times \R^+$ which is useful for studying the static beta coalescent that describes the genealogy of the population at time $0$.   { Writing $$r=-s$$ for the {\em reverse} time, which is the coalescence time direction,} we obtain $\Theta_n$ { by first restricting $\Psi_n$ to all points $(s,u)$ with $s\le 0$  and then applying the mapping} $(s,u) \mapsto (x, y)$,  $t \le 0$, to the remaining points,  where 
\begin{align}
x =  m(r), \hspace{.3in} y = m(r)u, 
\label{xymapping}
\end{align} 
for
\begin{equation}\label{mtdef}
m( r ) = \bigg( \frac{\alpha \Gamma(\alpha)}{r + \alpha \Gamma(\alpha)} \bigg)^{1/(\alpha - 1)}, \hspace{.3in} r \ge 0.
\end{equation} 
This quantity is the asymptotic proportion of the number of blocks that have not yet coalesced by the reverse time $r$; see Lemma \ref{Ntlem} below.
\begin{figure}[H]
\psfrag{p}{$\Psi_n |_{ \mathbb R^- \times \mathbb R^+}$}
\psfrag{t}{$t$}
\psfrag{u}{$u$}
\psfrag{z}{$\Theta_n$}
\psfrag{y}{$y$}
\psfrag{x}{$x$}
\psfrag{1}{$1$}
\psfrag{0}{$0$}
\psfrag{r}{$r$}
\hspace{1cm}\includegraphics[width=14cm]{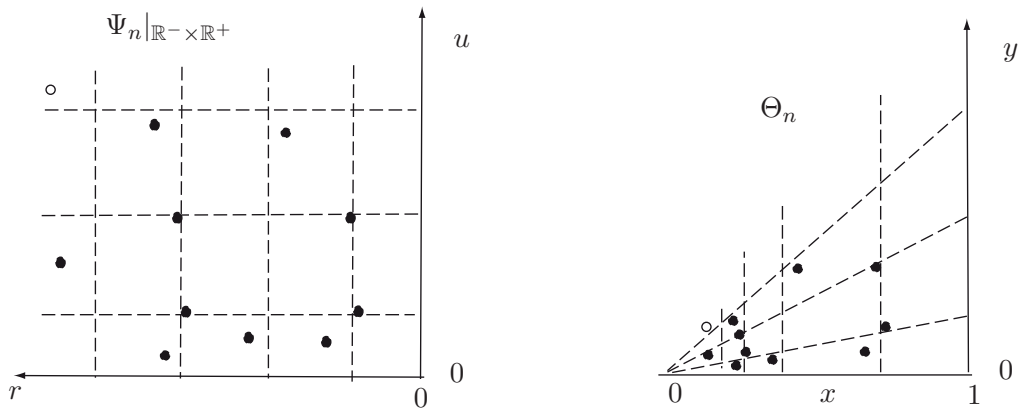}
\caption{The point process $\Theta_n$ arises from $\Psi_n |_{ \mathbb R^- \times \mathbb R^+}$ through the transformation $(-r,u)\mapsto (x,y) = (m(r), m(r)u)$. Here, $r=-s$, and $m(r)$ is given by formula  \eqref{mtdef}.}
\end{figure}
To calculate the intensity of $\Theta_n$, we invert the mapping given by \eqref{xymapping} to get $$r = \alpha \Gamma(\alpha) \bigg( \frac{1}{x^{\alpha - 1}} - 1 \bigg), \hspace{.3in} u = \frac{y}{x}.$$  It follows that
\[
\frac{ds}{dx}=-\frac{dr}{dx} =\alpha (\alpha - 1) \Gamma(\alpha) x^{-\alpha}, \hspace{.3in} \frac{du}{dy} = \frac{1}{x}.
\]
{Since $\frac {ds}{dy}=0$, the Jacobian determinant  of $(x,y)\mapsto (s,u)$ is $D(x,y)=
\alpha(\alpha-1)\Gamma(\alpha) x^{-\alpha-1}$}
and the intensity of $\Theta_n$ is
\begin{align}
\nu(dx,dy) =  \frac{1}{\Gamma(\alpha) \Gamma(2 - \alpha)} \bigg( \frac{y}x \bigg)^{-1-\alpha} D(x,y)\: dx\,dy  = dx \times \frac{\alpha (\alpha - 1)}{\Gamma(2 - \alpha)} y^{-\alpha-1} \: dy, 
\label{intensity}
\end{align}
which again does not depend on $n$.

Suppose $(x,y)$ is a point of $\Theta_n$.  Because the number of blocks in a beta coalescent at time $r$ is approximately $m( r ) n$, the merger corresponding to $(x,y)$ occurs when approximately a fraction $xn$ blocks remain, and the number of blocks that merge is approximately $n^{1/\alpha-1}\cdot yn=n^{1/\alpha} y$.  Therefore, whereas the second coordinate of $\Psi_n$ approximates the {\em fraction} of blocks lost due to a merger, the second coordinate of $\Theta_n$ represents the {\em number} of blocks lost due to a merger.  The fact that the Poisson process $\Theta_n$ is homogeneous in time reflects the fact that the distribution of the number of blocks lost in the first merger tends to a limit as the number of blocks at time zero tends to infinity; see section \ref{numblocksec}.

\subsection{Review of results on stable laws and Poisson integrals}\label{stablesec}

To state and prove our main results, we will need to review some results on both univariate and multivariate stable distributions.  We will restrict ourselves to the stable distributions of index $\alpha$, where $1 < \alpha < 2$.  Following the notation of \cite{sam}, we write in the univariate case $Z \sim S_{\alpha}(\sigma, \beta, \mu)$ if the characteristic function of the random variable $Z$ is given by
$$E(e^{i\theta Z}) = \exp \bigg(i\theta  \mu - \sigma^{\alpha} |\theta |^{\alpha} \bigg(1 - i \beta \mbox{sgn}(\theta ) \tan \bigg( \frac{\pi \alpha}{2} \bigg) \bigg) \bigg),$$
 with location parameter $\mu$ (which equals $E(Z)$ for $1<\alpha < 2$), scale parameter $\sigma > 0$, and skewness parameter $\beta \in [-1,1]$.  Here $\mbox{sgn}(\theta ) = 1$ if $\theta  > 0$, $\mbox{sgn}(\theta ) = -1$ if $\theta  < 0$, and $\mbox{sgn}(0) = 0$.  {We only deal with the case $\mu=0$.}

When $\mu=0$, $\beta =  1$, the characteristic function of $Z$ can also be written in the L\'evy-Khinchine form 
$$E(e^{i\theta Z}) = \exp \bigg(   b \int_0^{\infty} ( { e^{i\theta y} -1-i\theta y})   \: y^{-1-\alpha} \: dy \bigg)$$
(since $\alpha >1$, {we may and do avoid here the customary truncation within the integrand}), where
\begin{equation}\label{musigma}
 \sigma^\alpha = \frac{b \pi}{2\sin( \frac{\pi \alpha}{2} )\Gamma(\alpha+1)}.
\end{equation}

One can then construct $Z$ in the following way.  Consider a Poisson point process $\sum_{i \ge 1} \delta_{y_i}$  on $(0, \infty)$ with intensity $b y^{-1-\alpha} \: dy$, and for $\varepsilon >0$ define $$S(\varepsilon):= \sum_{y_i \ge \varepsilon} y_i, \quad Z(\varepsilon) = S(\varepsilon) - \frac{b\varepsilon^{1 - \alpha}}{\alpha - 1}.$$ Then the limit
\begin{equation}\label{Zlim}
Z = \lim_{\varepsilon \downarrow 0} Z(\varepsilon) \text{ a.s.}
\end{equation}
exists and obeys
\[
E(Z)=E(Z(\varepsilon)) = 0, \hspace{.3in} \textup{Var}(Z - Z(\varepsilon)) =  \int_0^{\varepsilon} y^2 \cdot b y^{-1-\alpha} \: dy = \frac{b \varepsilon^{2-\alpha}}{2 - \alpha}.
\]

To facilitate comparisons with results in \cite{ksw, kersting}, we note that if $c > 0$ and if $Z$ has a stable distribution such that
\begin{equation}\label{Zprop}
E(Z) = 0, \hspace{.3in} P(Z >z) = z^{-\alpha}(1+o(1)), \hspace{.3in} P(Z <-z ) =o( z^{-\alpha})
\end{equation}
as $z \to \infty$,  then $cZ \sim S_{\alpha}(\sigma, 1, 0)$ with $\sigma$ given by (\ref{musigma}) with $b = \alpha c^{\alpha},$ which means
\begin{equation}\label{sigmafromc}
\sigma = c \bigg( \frac{\pi}{2 \sin(\frac{\pi \alpha}{2}) \Gamma(\alpha)} \bigg)^{1/\alpha}.
\end{equation}
Note that the results in \cite{ksw, kersting} are stated with $-Z$ in place of $Z$.

Next we define the notion of an $\alpha$-stable random measure,  following sections 3.3 and 3.12 of \cite{sam}, again only for $\beta=1$.  Let $(E, {\cal E}, {\rho})$ be a measure space, and let ${\cal E}_0 = \{A \in {\cal E}: \rho(A) < \infty\}$.  Then an $\alpha$-stable random measure with control measure $\rho$ is defined as a countably additive function $M$ which assigns a random variable to each set $A \in {\cal E}_0$, satisfying the following properties:
\begin{enumerate}
\item If $A \in {\cal E}_0$, then $M(A) \sim S_{\alpha}(\rho(A)^{1/\alpha}, 1, 0)$.

\item If $A_1, \dots, A_k$ are disjoint sets in ${\cal E}_0$, then $M(A_1), \dots, M(A_k)$ are independent.
\end{enumerate}
To construct $M$, consider
a Poisson point process $\Xi$ on $E \times \R^+$ with intensity $a {\rho}(dx) \times y^{-1-\alpha} \: dy$, where 
 \begin{align}\label{aControl}
 a = \frac{2 \sin(\frac{\pi \alpha}{2}) \Gamma({\alpha +1})}{\pi}.
 \end{align}
Then the second coordinates of those points of $\Xi$ that fall in $A \times \R^+$ form a Poisson point process on $\R^+$ with intensity $a \rho(A) y^{-1-\alpha} \: dy$.  If one constructs $M(A)$ just as $Z$ is obtained in (\ref{Zlim}) with $b = a \rho(A)$, then $M=M^\Xi$ is an $\alpha$-stable random measure with control measure $\rho$ (and $\beta=1$) obtained from $\Xi$.

As noted in section 3.4 of \cite{sam}, if $f: E \rightarrow {\R}$ is a function such that $\int_E |f(x)|^{\alpha} \: \rho(dx) < \infty$, then one can define the integral 
 \begin{align}\label{stint}
 I(f)=\int_E f(x) \: M(dx)
 \end{align}
  by approximating $f$ with simple functions.  It is shown in \cite{sam} that $I(f) \sim S_{\alpha}(\sigma_f, \beta_f, 0)$, where 
\begin{align}\label{sigmaBeta}
\sigma_f = \bigg( \int_E |f(x)|^{\alpha} \: \rho(dx) \bigg)^{1/\alpha}, \hspace{.3in} \beta_f = \frac{\int_E  |f(x)|^\alpha\mbox{sgn}(f(x)) \: \rho(dx)}{\int_E |f(x)|^\alpha \: \rho(dx)}.
\end{align}

In particular it follows  for any linear combination
\[ \theta _1 I(f_1)+ \cdots + \theta _dI(f_d)= I(\theta _1f_1+ \cdots + \theta _d f_d)\sim S(\sigma_f,\beta_f,0) \]
with $f=\theta _1f_1+ \cdots + \theta _d f_d$. In the case $\alpha >1$ this implies that the joint distribution of $I(f_1), \ldots,I(f_d)$ is a multivariate $\alpha$-stable distribution. The characteristic function is
\begin{align*} E(&e^{i(\theta _1 I(f_1)+ \cdots + \theta _dI(f_d))}) \\ &= \exp\bigg(- \int_E|\theta _1f_1+ \cdots + \theta _d f_d|^\alpha (1- i\:\mbox{sgn}(\theta _1f_1+ \cdots + \theta _d f_d) \tan(\tfrac{\pi\alpha}2)) \bigg)\: d\rho \\
&= \exp\bigg(-\int_{S^{d-1}}|\theta \cdot s|^\alpha (1-i \:\mbox{sgn}(\theta \cdot s) \tan(\tfrac{\pi\alpha}2)) \bigg)\: \Gamma(ds)
\end{align*}
where $S^{d-1}\subset \mathbb R^d$ is the unit sphere, $\theta \cdot s= \theta _1s_1+ \cdots \theta _ds_d$, and $\Gamma(ds)$ is the finite measure on $S^{d-1}$ obtained from $(\sum_i f_i(y)^2)^{\alpha/2} \rho(dy)$ by the transformation $y \mapsto (f_1(y), \ldots, f_d(y))/(\sum_i f_i(y)^2)^{1/2}$; see section 3.2 of \cite{sam}.  Here $\Gamma$ is the so-called spectral measure of the $\alpha$-stable random vector $(I(f_1), \ldots, I(f_d))$, which characterizes the joint distribution.

In our case $E$ will be either $\mathbb R $ or the interval $(0,1]$. Then the above Poisson integrals can also be viewed as stochastic integrals using a L\'evy process $L=L^\Xi$ with mean zero, constructed from the Poisson point process $\Xi$ in the usual manner via compensation.  Thus,
\begin{align}\label{intL}\int_E f \: dM = \int_E f \; dL. 
\end{align}



\subsection{An asymptotic expansion for the static case}\label{resultsec}

Consider a beta coalescent back from time 0, and let $N_n(r)$ be the number of blocks in the partition at time $r$, so in particular $N_n(0)=n$.  Let $R_0 = 0$, and for $k \geq 1$, let $$R_k =R_{n,k}= \inf\{r> R_{k-1}: N_n( r ) \neq N_n(R_{k-1})\}.$$  Let $\tau_n = \max\{k: R_k < \infty\}$.  Thus, $R_1 < \dots < R_{\tau_n}$ are the times at which mergers occur, and $\tau_n$ is the number of mergers before only one block remains.  For $0 \leq k \leq \tau_n$, let 
$$X_k=X_{n,k} = N_n(R_k),$$
 and let $X_k = 1$ for $k > \tau_n$, which means $n = X_0 > X_1 > \dots > X_{\tau_n} = 1$.  The process $(X_k)_{k=0}^{\tau_n}$ is called the block-counting process associated with the beta coalescent.  


The result below shows that certain functionals of the block-counting process have an asymptotic stable law as $n \rightarrow \infty$. For this purpose, we consider the beta coalescent, constructed from $\Upsilon_n$ as in Section \ref{constsec}, and the $\alpha$-stable random measures $M_n=M^{\Theta_n}$ obtained from Poisson point processes $\Xi_n$ as described in Section \ref{stablesec}, with $E=(0,1]$, intensity \eqref{intensity} and $\Xi_n=\Theta_n$ (arising from $\Upsilon_n$ as described in Section \ref{twoPPP}). In view of \eqref{aControl} the control measure of $M_n$ is
\begin{align}\label{contmeas}
 \rho(dx)= \frac{\pi}{2 \sin(\frac{\pi \alpha}{2}) \Gamma({\alpha + 1})} \cdot\frac{\alpha (\alpha - 1)}{\Gamma(2 - \alpha)} \: dx = \frac{\pi (\alpha - 1)}{2 \sin(\frac{\pi \alpha}{2}) \Gamma(\alpha) \Gamma(2 - \alpha)}.
 \end{align}
We denote by $\mathscr F$ the set of all differentiable functions $f$ which for some $c > 0$ and $0 < \zeta < \frac{1}{\alpha}$ obey 
\begin{align}|f'(x)| \leq c x^{-\zeta - 1}
\label{fcond}
\end{align} 
for all $x \in (0, 1]$.
\begin{Theo}\label{onetime}
Let $(X_k)_{k=0}^{\tau_n}$ be the block-counting process associated with the beta coalescent, and 
let $f\in \mathscr F$. Then as $n \to \infty$,
\begin{align} \label{asex}
 \frac{1}{\alpha - 1} \sum_{k < \tau_n} f \bigg( \frac{X_k}{n} \bigg)  = n \int_0^1 f(x) \: dx  - n^{1/\alpha} \int_{(0,1]} f(x) \: M_n(dx) + o_P(n^{1/\alpha}),
  \end{align}
  where $M_n$ is an $\alpha$-stable random measure with control measure \eqref{contmeas}.
\end{Theo}
The statement of Theorem \ref{onetime} is an asymptotic version of the equation
\begin{align}\label{sums}
\gamma \sum_{k < \tau_n} f \bigg( \frac{X_k}{n} \bigg)  = n \sum_{k < \tau_n} f \bigg( \frac{X_k}{n} \bigg) \frac{Y_k}{n} - n^{1/\alpha}\sum_{k < \tau_n} f \bigg( \frac{X_k}{n} \bigg)\frac{Y_k-\gamma}{n^{1/\alpha}} , 
\end{align}
where $\gamma = 1/(\alpha-1)$ and $Y_k=X_k-X_{k+1}$, $0\le k < \tau_n$. The proof consists of showing that the  sums appearing in the right-hand side of \eqref{sums} may asymptotically be replaced by the  integrals $ \int_0^1 f(x) \: dx$ and $\int_{(0,1]} f(x) \: M_n(dx)$, respectively.
\\

Since the control measure, and hence the distribution, of $M_n$
does not depend on~$n$, the asymptotic expansion \eqref{asex} in Theorem~\ref{onetime} directly imples convergence in distribution not only for a single $f \in  \mathscr F$, but also for finitely many $f_1,\ldots, f_d$.  To ease notation, let us define, for $f \in \mathscr F$ and $n \in \mathbb N$,\begin{align}\label{defJ}
\mathcal J_n(f) := n^{-1/\alpha} \bigg( \frac{1}{\alpha - 1} \sum_{k < \tau_n} f \bigg( \frac{X_k}{n} \bigg) - n \int_0^1 f(x) \: dx \bigg).
\end{align}
\begin{Cor} Let $f, f_1,\ldots, f_d \in \mathscr F$. Then
$$\mathcal J_n(f) \Rightarrow  S_{\alpha}(\sigma_f, -\beta_f, 0)$$
and $$(\mathcal J_n(f_1),\ldots, \mathcal J_n(f_d)) \Rightarrow ( - I(f_1), \ldots, - I(f_d)),$$
where $\sigma_f, \beta_f$ are as in~\eqref{sigmaBeta}, $I(f)$ is as in  \eqref{stint} with  an $\alpha$-stable random measure $M$ with control measure~\eqref{contmeas}, and $\Rightarrow$ denotes convergence in distribution as $n \rightarrow \infty$.
\end{Cor}

\subsection{The evolving beta coalescent}\label{secev}

Now we transform the asymptotic expansion from Theorem \ref{onetime} to the evolving coalescent. Here it is convenient to use L\'evy processes.  Let
\begin{align}\label{defL}
L_n=(L_{n,s})_{-\infty < s < \infty} := L^{\Psi_n}  , \hspace{.3in} L_n' =(L_{n,x}')_{0\le x \le 1}:= L^{\Theta_n}
\end{align}
be the mean zero stable L\'evy processes with $L_{n,0}=L_{n,0}=0$ a.s. and with jumps $\Delta L_{n,s}=u$ and $\Delta L_{n,x}'=y$ for all points $(s,u)$ from $\Psi_n$ and $(x,y)$ from $\Theta_n$, respectively. For the time-reversals we write
\begin{align*}\hat L_{n,r} &:= L_{n,s-}\,, \quad s=-r,\, r \ge 0,\\
 \hat L'_{n,w} &:= L'_{n,x-}-L_{n,1}, \quad w=1-x\,, \, 0< x \le 1.
 \end{align*}
Then the mapping from \eqref{xymapping} translates into the following lemma.
\begin{Lemma}\label{Levysub}
Let $f:(0,1]\to \mathbb R$ be such that $\int_0^1 |f(x)|^\alpha\: dx < \infty$. Then
$$\int_0^\infty f(m(r)) m(r) \:d\hat L_{n,r} = \int_0^1 f(1-w)\:d\hat L'_{n,w}.$$
\end{Lemma}

\begin{proof}
The assumption on $f$ guarantees that the integrals are well-defined. The processes $(\hat L_{n,r})_{r\ge 0}$ and $(\hat L'_{n,m(r)})_{r\ge 0}$ are  mean zero L\'evy processes, and hence   martingales with respect to the filtration $$\mathcal F_r := \sigma (\Psi_n \mid_{[-r,0]\times \mathbb R^+}) = \sigma(\Theta_n \mid_{[m(r),1]\times \mathbb R^+})\, , \quad r\ge 0.$$
Consequently
\begin{align*}  J_r:= \int_0^r f(m(q)) m(q) \: d\hat L_{n,q}-\int_0^{1-m(r)} f(1-w)\: d\hat L'_{n,w}\, , \quad r\ge 0,
\end{align*}
is a local martingale. Jumps can only occur at points $(r,u)$ in $\Psi_n$; because of \eqref{xymapping} they vanish:
$$\Delta J_r = -f(m(r))m(r) u+f(x)y = 0.$$
Thus, $J$ is a.s. continuous. Moreover, since the underlying processes are L\'evy processes without a Brownian component,  the quadratic variation of $J$ is $[J]_\infty = \sum_{r\ge 0} (\Delta J_r)^2 = 0$ a.s. Thus $J_\infty = J_0 = 0$ a.s. 
\end{proof}

From \eqref{defL} and \eqref{intL} it follows that $M_n=M^{\Theta_n}$ satisfies
\begin{align*}
\int_{(0,1]} f(x)\: M_n(dx)= \int_0^1 f(x)\; d L'_{n,x} = -\int_0^1 f(1-w)\; d\hat L'_{n,w}.
\end{align*}
Now we apply Lemma \ref{Levysub} to get
\begin{align}
\int_{(0,1]} f(x)\: M_n(dx)= -\int_0^\infty f(m(r))m(r)\; d L_{n,-r}.
\label{ML}
\end{align}

Let us now proceed to consider the evolving coalescent  $(\mathcal T_n(t), t \in \mathbb R)$ described in Section \ref{constsec}.  For each  $s \in \mathbb R$ and $n \in \mathbb N$, we denote the block counting process of  the coalescent tree $\mathcal T_n(n^{1-\alpha}s)$ by $(X_k^s)_{k=0}^{\tau_n^s}$. 
By shifting the origin of the scaled time to the time point $s$ and re-centering the process $L$ at this new time origin (which does not affect its increments), we can apply Theorem~\ref{onetime} together with \eqref{ML} and conclude that
$$\frac{1}{\alpha - 1} \sum_{k < \tau_n^s} f \bigg( \frac{X_{k}^s}{n} \bigg)  = n \int_0^1 f(x) \: dx  + n^{1/\alpha} \int_{0}^\infty f(m(r))m(r)\; dL_{n,s-r} + o_P(n^{1/\alpha}).
$$
Writing $\mathcal J_f^n(s)$ for the random variable \eqref{defJ} with $(X_k)=(X_k^0)$ replaced by $(X_k^s)$, we thus obtain
\begin{align*} 
\mathcal J_f^n(s) 
= \int_{0}^{\infty} f(m(r)) m(r)\: dL_{n,s-r} + o_P(1).
\end{align*} 
Since the distribution of the Poisson point process $\Psi_n$, and hence also that of  the L\'evy process $L_n = L^{\Psi_n}$, does not depend on $n$, we obtain the following result for the evolving beta coalescent.
\begin{Cor} \label{Cor2}
For $f \in \mathscr F$ and $s \in \mathbb R$, let $\mathcal J_{n,s}(f)$ be as in  \eqref{defJ}, but now evaluated at the coalescent tree $\mathcal T_n(n^{1-\alpha}s)$ instead of $\mathcal T_n(0)$. Then the sequence of stationary processes $(\mathcal J_{n,s}(f))_{-\infty < s < \infty}$,  $n \ge 1$,
converges as $n \to \infty$ in finite-dimensional distributions  to the moving average process
\[  \int_{0}^{\infty} f(m(r)) m(r)\: dL_{s-r} , \hspace{.3in}  -\infty < s < \infty, \]
where $m$ is given by  \eqref {mtdef} and and $(L_s)_{-\infty < s < \infty }$ is a mean zero  L\'evy process with $L_0 =0$ and L\'evy measure given by \eqref{Levymeas}.
\end{Cor}


To understand better how these functionals of the beta coalescent evolve over time, note that the stable random variable $\int_{(0,1]} f(x) \: M_n(dx)$ from Theorem \ref{onetime}, which gives the limit of the functional $\mathcal J_{n,0}(f)$, is a function of the Poisson process $\Theta_n$, and is therefore also a function of the Poisson process $\Psi_n$.  Likewise, the stable random variable that gives the limit of $\mathcal J_{n,s}(f)$ can be expressed as a function of a Poisson process $\Theta_n^s$ and as a function of a Poisson process $\Psi_n^s$.  The Poisson process $\Psi_n^s$ can be obtained from $\Psi_n$ by a simple time shift.  If $(-r, u)$ is a point of $\Psi_n$, then $(-r -s , u)$ is a point of $\Psi_n^s$.  We obtain the Poisson process $\Theta_n^s$ by applying the transformation $(s-r, u) \mapsto (x, y) = (m(r), m(r)u)$ to the points of $\Psi_n|_{(-\infty, s]}$, where $m(r)$ is again defined by \eqref{mtdef}.  See Figure 3 for an illustration of how these point process evolve as $s$ increases.
\begin{center}
\begin{figure}[H]\label{shift}
\psfrag{y}{$\Upsilon'$}
\psfrag{0}{$t=0$}
\psfrag{2}{$0$}
\psfrag{t}{$t$}
\psfrag{u}{$u$}
\psfrag{p}{$p$}
\psfrag{s}{$s$}
\psfrag{1}{$p=1$}
\psfrag{q}{$\Psi_n$}
\psfrag{t}{$t$}
\psfrag{u}{$u$}
\psfrag{z}{$\Theta_n^0$}
\psfrag{a}{$\Theta_n^s$}
\psfrag{y}{$y$}
\psfrag{x}{$x$}
\psfrag{1}{$1$}
\psfrag{0}{$0$}
\psfrag{r}{$r$}
\hspace{1cm}\includegraphics[height=11cm, width=12cm]{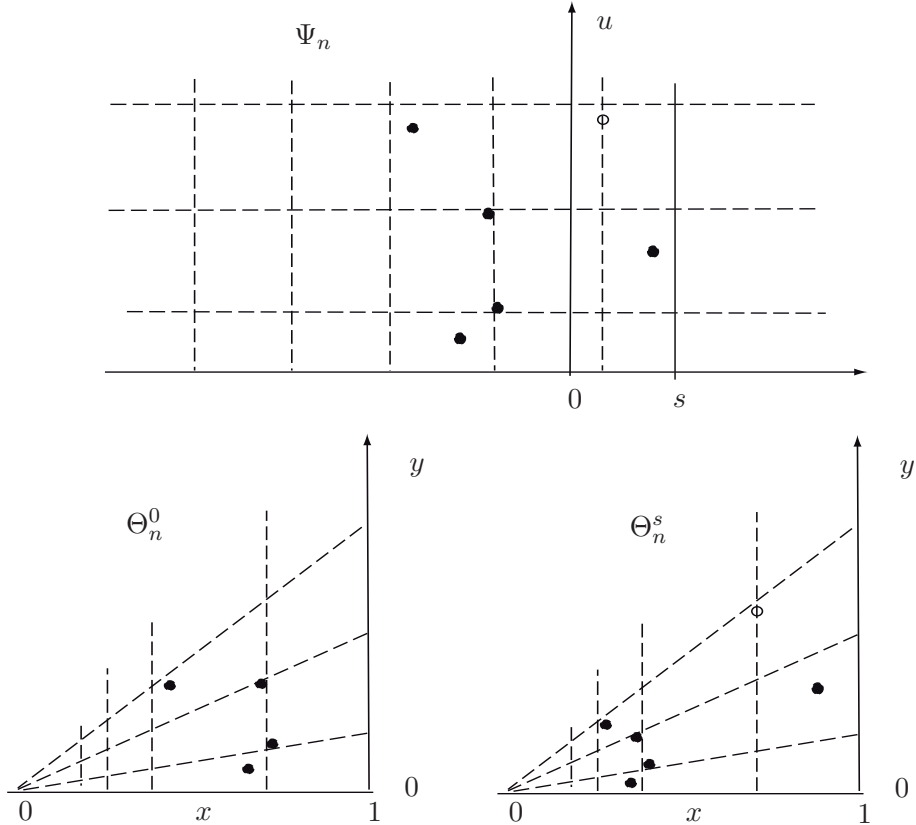}
\caption{The point process $\Theta_n^s$ arises from $\Psi_n |_{ (-\infty,s] \times \mathbb R^+}$ through the transformation $(s-r,u)\mapsto (x,y) = (m(r), m(r)u)$. As $s$ increases above $0$, the points of $\Theta_n^0$ wander down to the left towards $(0,0)$, and new points wander in from the right.}
\end{figure}
\end{center}

\subsection{Functionals of the beta coalescent}\label{Fubecoa}

In this section, we consider three functionals of the beta coalescent: the number of collisions, the total branch length, and the total length of external branches.  We observe how Theorem~\ref{onetime} allows us to recover known results for the asymptotic distributions of these quantities for the static beta coalescent.  Then Corollary \ref{Cor2} allows us to describe how these functionals behave over time in the evolving beta coalescent.  We also obtain a new result about the ratio of the external branch length to the total branch length, which could be of interest for biological applications.

\begin{Exm}{\em Consider the number $\tau_n$ of collisions before just a single block remains.  Because $$\tau_n = \sum_{k < \tau_n} 1 = \frac{1}{\alpha - 1} \sum_{k < \tau_n} (\alpha - 1),$$ we can apply directly the result of Theorem \ref{onetime} with $f(x) = \alpha - 1$ for all $x \in (0,1]$.  We get that for $1 < \alpha < 2$,
\begin{align}
n^{-1/\alpha}(\tau_n - (\alpha - 1) n) \Rightarrow S(\sigma_1,-1,0),
\label{taunlim}
\end{align}
where $$\sigma_1 = \bigg( \frac{\pi (\alpha - 1)^{1 + \alpha}}{2 \sin( \frac{\pi \alpha}{2}) \Gamma(\alpha) \Gamma(2 - \alpha)} \bigg)^{1/\alpha}.$$  This agrees with the result of Lemma 4 in \cite{ksw}, where the limit on the right-hand side of (\ref{taunlim}) is expressed as $c_1 Z$ for $-Z$ satisfying (\ref{Zprop}) and $c_1 = (\alpha - 1)^{1 + 1/\alpha}/\Gamma(2 - \alpha)^{1/\alpha}$.  This result had also been shown in \cite{ddsj, gy, im}, and the equivalence between the two ways of expressing the limit can be seen from (\ref{sigmafromc}).

Because we use this result in our proof of Theorem \ref{onetime}, we have not obtained here another independent proof of this result. The benefit is that our approach allows us to examine the common distribution of $\tau_n$ and other functionals.  Also, Corollary \ref{Cor2} with $f(x) = \alpha - 1$ allows us to understand how the total number of collisions changes over time for the evolving beta coalescent.  In particular, we see that the limit process is a stationary stable process that can be expressed, in a relatively simple way, as a moving average process.}
\end{Exm}

\begin{Exm}\label{Lnex}
{\em Consider next the total length $\mathcal L_n$ of all branches in the coalescent tree.  This quantity is of interest in Biology because the total branch length should be approximately proportional to the number of mutations observed in a sample of $n$ individuals.  Note that $$\mathcal L_n = \sum_{k < \tau_n} X_k (R_{k+1} - R_k).$$  Define also $$\mathcal L_n' = \sum_{k < \tau_n} \frac{X_k}{\lambda_{X_k}}, \hspace{.3in} \mathcal L_n'' = \alpha \Gamma(\alpha) \sum_{k < \tau_n} X_k^{1 - \alpha}.$$  Lemma 2.2 in \cite{ddsj} implies that as $m \rightarrow \infty$, 
\begin{align}\label{rate}
\lambda_m = \frac{1}{\alpha \Gamma(\alpha)} m^{\alpha} + O(m^{\alpha - 1}).
\end{align}  
Therefore, there is a constant $c > 0$ such that
\begin{equation}\label{Lnprime}
|\mathcal L_n' - \mathcal L_n''| \leq \sum_{k < \tau_n} X_k \bigg| \frac{1}{\lambda_{X_k}} - \frac{\alpha \Gamma(\alpha)}{X_k^{\alpha}} \bigg| \leq\sum_{m=1}^n m \bigg| \frac{1}{\lambda_m} - \frac{\alpha \Gamma(\alpha)}{m^{\alpha}} \bigg| \leq \sum_{m=1}^n m \cdot c m^{-1-\alpha} = O(1).
\end{equation}
Also, conditional on $\sigma(X) = \sigma(X_0, X_1, \dots, X_{\tau_n})$, the distribution of $X_k(R_{k+1} - R_k)$ is exponential with rate parameter $\lambda_{X_k}$.  It follows that
\begin{equation}\label{mean2prime}
E(\mathcal L_n - \mathcal L_n'\mid \sigma(X)) = 0
\end{equation}
and
\begin{equation}\label{var2prime}
\mbox{Var}(\mathcal L_n - \mathcal L_n'\mid \sigma(X)) = \sum_{k=0}^{n-1} X_k^2 1_{k < \tau_n} \cdot \frac{1}{\lambda_{X_k}^2} \leq \sum_{m=1}^n \frac{m^2}{\lambda_m^2} = O(1 \vee n^{3 - 2 \alpha}).
\end{equation}
It now follows from (\ref{Lnprime}), (\ref{mean2prime}), (\ref{var2prime}), and Chebyshev's Inequality that if $1 < \alpha < \frac{1}{2}(1 + \sqrt{5})$, so that $1 + \alpha - \alpha^2 > 0$, we have $$n^{\alpha - 1 - 1/\alpha} (\mathcal L_n - \mathcal L_n'') \Rightarrow 0.$$
Therefore, we may replace $\mathcal L_n$ by $\mathcal L_n''$ in asymptotic calculations.  Because $$n^{\alpha - 1} \mathcal L_n'' = \frac{1}{\alpha - 1} \sum_{k < \tau_n} \alpha (\alpha - 1) \Gamma(\alpha) \bigg( \frac{X_k}{n} \bigg)^{1-\alpha},$$ when $1 < \alpha < \frac{1}{2}(1 + \sqrt{5})$ we can apply Theorem \ref{onetime} with $f(x) = \alpha (\alpha - 1) \Gamma(\alpha) x^{1 - \alpha}$.  We get 
$$n^{\alpha - 1 - 1/\alpha} \bigg( \mathcal L_n - \frac{\alpha (\alpha - 1) \Gamma(\alpha) n^{2 - \alpha}}{2 - \alpha} \bigg) \Rightarrow  S_{\alpha}(\sigma_2, -1, 0),$$ 
where 
$$\sigma_2 = \bigg( \frac{\pi \alpha^{\alpha} (\alpha - 1)^{1 + \alpha} \Gamma(\alpha)^{\alpha - 1}}{2 \sin( \frac{\pi \alpha}{2}) \Gamma(2 - \alpha) (1 + \alpha - \alpha^2)} \bigg)^{1/\alpha},$$
which agrees with part (i) of Theorem 1 in \cite{kersting}.  It also follows from Corollary \ref{Cor2} that for the evolving beta coalescent, the evolution of the total branch length, scaled as above, converges in the sense of finite-dimensional distributions to a stationary stable process.}
\end{Exm}

\begin{Exm}\label{lnex}
{\em Consider also the total length $\ell_n$ of all external branches in the tree.  This quantity is also of interest in Biology, as it should be approximately proportional to the number of mutations that appear on just one individual in a sample of $n$ individuals.  It is shown in the proof of Theorem 1 in \cite{ksw} that $$\ell_n = \alpha (\alpha - 1)^2 \Gamma(\alpha) n^{2 - \alpha} + \alpha (2 - \alpha) \Gamma(\alpha) n^{1-\alpha} \tau_n + o_P(n^{1 + 1/\alpha - \alpha}).$$  Therefore, $$n^{\alpha - 1} \ell_n - \alpha (\alpha - 1)^2 \Gamma(\alpha) n = \frac{1}{\alpha - 1} \sum_{k < \tau_n} \alpha (\alpha - 1) (2 - \alpha) \Gamma(\alpha) + o_P(n^{1/\alpha}),$$ so for $1 < \alpha < 2$, we can apply Theorem \ref{onetime} with $f(x) = \alpha (\alpha - 1) (2 - \alpha) \Gamma(\alpha)$ to get
$$n^{\alpha - 1 - 1/\alpha} (\ell_n - \alpha (\alpha - 1) \Gamma(\alpha) n^{2-\alpha}) \Rightarrow S_{\alpha}(\sigma_3, -1, 0),$$ where  $$\sigma_3 = \bigg( \frac{\pi \alpha^{\alpha} (\alpha - 1)^{1 + \alpha} (2 - \alpha)^{\alpha} \Gamma(\alpha)^{\alpha - 1}}{2 \sin(\frac{\pi \alpha}{2}) \Gamma(2 - \alpha)}\bigg)^{1/\alpha},$$ in agreement with Theorem 1 of \cite{ksw}.}
\end{Exm}

\begin{Exm}\label{lnLn}
{\em Finally, we consider the quantity $\ell_n/\mathcal L_n$, which in the biological setting should be approximately equal to the proportion of mutations that appear on only one individual.  This ratio is potentially useful for drawing inferences about the genealogy of a population from data, in part because the value that we expect for this ratio does not depend on the mutation rate, which is often unknown.  Indeed, it follows from results in \cite{bbs2} that the parameter $\alpha$ in the beta coalescent can be consistently estimated by the quantity $2 - \ell_n/\mathcal L_n$.

Assume that $1 < \alpha < \frac{1}{2} (1 + \sqrt{5})$.  From the discussion in Examples \ref{Lnex} and \ref{lnex}, we see that $$n^{\alpha - 1 - 1/\alpha} \mathcal L_n = \frac{\alpha (\alpha - 1) \Gamma(\alpha)}{2 - \alpha} n^{1-1/\alpha} - Z_1 + o_P(1),$$ where $$Z_1 = \int_0^1 \alpha (\alpha - 1) \Gamma(\alpha) x^{1-\alpha} \: M_n(dx),$$ and likewise $$n^{\alpha - 1 - 1/\alpha} \ell_n = \alpha (\alpha - 1) \Gamma(\alpha) n^{1 - 1/\alpha} - Z_2 + o_P(1),$$ where $$Z_2 = \int_0^1 \alpha (\alpha - 1)(2 - \alpha) \Gamma(\alpha) \: M_n(dx).$$  Therefore,
\begin{align*}
\frac{\ell_n}{\mathcal L_n} &= \frac{\alpha (\alpha - 1)\Gamma(\alpha) - n^{-1 + 1/\alpha} Z_2 + o_P(n^{-1 + 1/\alpha})}{\frac{\alpha (\alpha - 1)\Gamma(\alpha)}{2-\alpha} - n^{-1 + 1/\alpha} Z_1 + o_P(n^{-1 + 1/\alpha})}   \\
&= (2 - \alpha) + n^{-1 + 1/\alpha} \bigg( \frac{(2 - \alpha)^2}{\alpha (\alpha - 1)\Gamma(\alpha)} Z_1 -\frac{2 - \alpha}{\alpha (\alpha - 1)\Gamma(\alpha)} Z_2  \bigg) + o_P(n^{-1 + 1/\alpha}) \\
&=    (2 - \alpha) + n^{-1 + 1/\alpha}  \int_0^1 (2 - \alpha)^2(x^{1-\alpha} - 1) \: M_n(dx) + o_P(n^{-1 + 1/\alpha}).
\end{align*}
It follows that $$n^{1 - 1/\alpha} \bigg( \frac{\ell_n}{\mathcal L_n} - (2 - \alpha) \bigg) \Rightarrow S_{\alpha}(\sigma_4, 1, 0),$$ where 
$$\sigma_4 = (2 - \alpha)^2 \bigg( \frac{\pi (\alpha - 1)}{2 \sin(\frac{\pi \alpha}{2}) \Gamma(\alpha)\Gamma(2 - \alpha)} \int_0^1 (x^{1-\alpha} - 1)^{\alpha} \: dx \bigg)^{1/\alpha}.$$
Using the substitution $y = x^{\alpha-1}$, the integral transforms to a beta integral:
$$\int_0^1 (x^{1-\alpha} - 1)^{\alpha} \: dx = \frac 1{\alpha-1}\int_0^1 (1-y)^\alpha y^{\frac{2-\alpha^2}{\alpha-1}}dy = \frac{\Gamma(\alpha+1) \Gamma(\frac{\alpha+1-\alpha^2}{\alpha-1})}{(\alpha-1)\Gamma(\frac \alpha{\alpha-1})}.$$
Altogether,
$$\sigma_4 = (2 - \alpha)^2 \bigg( \frac{\pi \alpha}{2 \sin(\frac{\pi \alpha}{2}) } \frac{ \Gamma(\frac{\alpha+1-\alpha^2}{\alpha-1})}{\Gamma(2 - \alpha)\Gamma(\frac \alpha{\alpha-1})}\bigg)^{1/\alpha}.$$
If $\frac{1}{2}(1 + \sqrt{5}) \leq \alpha < 2$, then the fluctuations in $\mathcal L_n$ are of a higher order of magnitude than the fluctuations of $\ell_n$, so the asymptotic distribution of $\ell_n/\mathcal L_n$ is determined by the asymptotics of $\mathcal L_n$ given in Theorem 2 of \cite{ksw}.  In particular, when $\frac{1}{2}(1 + \sqrt{5}) < \alpha < 2$, the asymptotic distribution of $\ell_n/\mathcal L_n$ is no longer a stable law.}
\end{Exm}

\section{Proofs}\label{proofsec}

Let us remark in advance that for $f$ satisfying \eqref{fcond} with $\zeta < 2$ by linearity  we may and will assume the following properties:
$f(x) \ge 1$ for all $x\in (0,1]$, $f(x)$ is monotonically decreasing and $x^2f(x)$ is monotonically increasing. 

Indeed, for any $f$ satisfying the condition $|f'(x)| \le c x^{-\zeta -1}$, $x \in (0,1]$, we may write $f=f_1-f_2$ with
\[ f_1(x)= f(x) + 3c\zeta^{-1}x^{-\zeta} + d, \hspace{.3in} f_2(x)= 3c\zeta^{-1}x^{-\zeta} + d .\]
Then $f_1$ and $f_2$ fulfil these three requirements, if we let $d= 1- f(1)\wedge 0$. For $f_2$ this is obvious, since $\zeta < 2$. Furthermore $ f_1'(x) \le - 2cx^{-\zeta-1}$. This implies that $f_1$ is decreasing and also $f_1(x) \ge 1$ for all $x$, since $f_1(1)\ge 1$. Moreover $d+f(1)\ge 0$, thus $f_1(1) \ge   3c\zeta^{-1} $ and 
\[f_1(x)= f_1(1)-\int_x^1 f_1'(y)\: dy \ge 2c\zeta^{-1}+\int_x^1 2cy^{-\zeta-1}\: dy =  2c\zeta^{-1} x^{-\zeta}\ge 2cx^{-\zeta}.\]
Taking also into account $f_1'(x) \ge -4c x^{-\zeta-1}$ we obtain
\[ \frac{d}{dx} x^2f_1(x) = 2xf_1(x)+ x^2 f_1'(x) \ge 0. \]
This gives the assertion.

Note also for $f$ satisfying \eqref{fcond}, there is a positive constant $c$ such that
\begin{equation}\label{fbound}
f(x) \leq c x^{-\zeta}
\end{equation}
for all $x \in (0, 1]$ and therefore $$\int_0^1 f(x)^{\alpha} \: dx < \infty.$$
To facilitate notation we adopt, here and throughout the rest of the paper, the convention that $c>0$ denotes a constant, only dependent on $\alpha$, which may change its value from term to term.

\subsection{The number of blocks for the beta coalescent}\label{numblocksec}

We assemble here some results about the evolution of the number of blocks for the beta coalescent.  We adopt the notation of Section \ref{resultsec}, so that $\tau_n$ is the total number of mergers and $(X_k)_{k=0}^{\tau_n}$ is the block counting process. Let $$Y_k = X_k - X_{k+1},\hspace{.3in} k \ge 0,$$ which are the numbers of blocks lost during the  mergers.  

Define
\begin{align*} q_i= \frac{\alpha}{\Gamma(2-\alpha)} \frac{\Gamma(i+1-\alpha)}{\Gamma(i+2)}, \quad i \ge 1. 
\end{align*}
The numbers $q_i$ are the weights of a probability distribution on $\mathbb N$ (see \cite{kersting}).  
From Stirling's formula
\begin{align}\label{Stir}
q_i  = \frac \alpha{\Gamma(2-\alpha)} i^{-\alpha -1}(1+o(1))
\end{align}
for $i \to \infty$. Let
\begin{equation}\label{pkmean}
\gamma = \sum_{i \ge 1} iq_i = \frac{1}{\alpha - 1},
\end{equation}
where the last equality is formula (5) in \cite{kersting}.  See also \cite{beleg, ddsj} for a discussion of this probability distribution.  

It has been known since the work of Bertoin and Le Gall \cite{beleg} that the distribution of the random variables $Y_0, Y_1, \dots$ is well approximated by $(q_i)_{i=1}^{\infty}$.  The next result gives a bound on the accuracy of this approximation.

\begin{Lemma}\label{UgivenX}
There is a number $c<\infty$ such that for $j \ge 2$ and $1\le i< j$
\[ \big| P(Y_0=i \mid X_0=j) - q_i \big| \le \frac{c i q_i}{j} \]
and thus
\[ P(Y_0=i \mid X_0=j)  \le (1+c) q_i. \]
\end{Lemma}

\begin{proof}
In the proof of Lemma 3 in \cite{kersting} (see there the two displayed formulas before (9)) it is shown that there are real numbers $b_j$ such that for $1\le i< j$
\[ \Big(1- \frac ij\Big) q_i \le b_j  P(Y_0=i \mid X_0=j)  \le q_i, \quad 1 - \frac \gamma j \le b_j \le 1 , \]
and so for $j> \gamma$
\[ \Big(1- \frac ij\Big) q_i \le    P(Y_0=i \mid X_0=j) \le \frac 1{1- \frac \gamma j} q_i. \]
This gives our claim in the case when $j> 2\gamma$. The other finitely many cases are covered too, if we choose $c$ sufficiently large.
\end{proof}

The next two lemmas contain our first applications of these estimates.

\begin{Lemma}\label{Xratio} For $n \to \infty$
\[ \max_{k < \tau_n} \frac{X_k}{X_{k+1}} = O_P(1) .\]
\end{Lemma}

\begin{proof} Let $a>1$ and $\eta=(a-1)/a$. Because $Y_k = X_k - X_{k+1},$
\[ P\bigg( \max_{k < \tau_n} \frac{X_k}{X_{k+1}} >a \bigg) = P\bigg( \max_{k<\tau_n} \frac{Y_k}{X_k}>\eta\bigg) \le E\bigg(\sum_{k<\tau_n} 1_{\{Y_k> \eta X_k\}}\bigg).\]
From Lemma \ref{UgivenX} and \eqref{Stir}
\[ P(Y_0 >\eta j \mid X_0=j) \le c \sum_{\eta j< i < j} i^{-\alpha-1} \le c((\eta j)^{-\alpha}-j^{-\alpha}) \]
and consequently, using that $\tau_n \leq n-1$,
\[P\bigg( \max_{k < \tau_n} \frac{X_k}{X_{k+1}} > a \bigg)\le \sum_{k=0}^{n-2} P(Y_k> \eta X_k) \le c(\eta^{-\alpha}-1) \sum_{k=0}^{n-2} E(X_k^{-\alpha}) \le c (\eta^{-\alpha}-1) \sum_{j=1}^\infty \frac 1{j^\alpha}. \]
Since $\alpha>1$, the series is convergent. Also as $a \to \infty$, we have $\eta \to 1$, and the claim follows.
\end{proof}

Next let us introduce the stopping times
\[ \tau_n(a) := \min\{k \ge 0: X_k \le an\}, \quad 0< a \le 1.\]
In particular, $\tau_n(1)=0$.

\begin{Lemma}\label{Xstop} For $n \to \infty$
\begin{align*} \tau_n(a)= \frac {(1-a)n}\gamma + o_P(n).
\end{align*}
\end{Lemma}

\begin{proof}
Let $\xi\in(\alpha^{-1},1)$. Then for $k < \tau_n$ from Lemma \ref{UgivenX} and \eqref{Stir}
\[ P(Y_k> n^\xi \mid X_k) \le c \sum_{i> n^\xi} i^{-\alpha-1} \le c (n^\xi)^{-\alpha} \]
and since $\tau_n \le n - 1$
\[ P\big(\max_{k< \tau_n} Y_k >n^\xi) \le \sum_{k<n-1} P(Y_k> n^\xi ) \le cn^{1-\xi\alpha} =o(1)\]
or $\max_{k<\tau_n} Y_k =O_P(n^\xi)$. Thus $X_{\tau_n(a)-1}-X_{\tau_n(a)}=Y_{\tau_n(a)-1}=O_P(n^\xi)$. Also by definition $X_{\tau_n(a)}\le an < X_{\tau_n(a)-1}$. Since $\xi < 1$ this gives
\[ n^{-1}X_{\tau_n(a)} \to a \]
in probability.




Observe that $\tau_n-\tau_n(a)= \tau_{X_{\tau_n(a)}}$, and recall that it was shown in  \cite{ddsj, gy, im} that $$\tau_n=n/\gamma +o_P(n).$$ Therefore, using the strong Markov property, we obtain
\[ \tau_n-\tau_n(a) = \frac{X_{\tau_n(a)}}\gamma+ o_P(X_{\tau_n(a)} )= \frac {an}\gamma + o_P(n), \]
which gives the claim.
\end{proof}

The next lemma pertains to the evolution of the number of blocks in continuous time and follows fairly directly from results in \cite{bbs}, where the number of blocks was studied for the beta coalescent started with infinitely many blocks.  Recall the definition of $m( r )$ from \eqref{mtdef} and let as above $N_n( r )$ denote the number of blocks at time $r$.

\begin{Lemma}\label{Ntlem}
Consider the beta coalescent started with $n$ blocks at time zero. Let $\varepsilon > 0$.  Suppose $h: \N \rightarrow (0, \infty)$ is a function such that $\lim_{n \rightarrow \infty} n^{1-\alpha} h(n) = 0$.  Then $$\lim_{n \rightarrow \infty} P \big( (1 - \varepsilon) m( r ) n \leq N_n(n^{1-\alpha} r) \leq (1 + \varepsilon) m( r ) n \mbox{ for all } r \in [0, h(n)] \big) = 1.$$
\end{Lemma}

\begin{proof}
Consider a beta coalescent started with infinitely many blocks at time zero, and let $N( r )=N_\infty( r )$ denote the number of blocks at time $r$.  Theorem 1.1 of \cite{bbs} states that
\begin{equation}\label{bbseq}
\lim_{r \downarrow 0} r^{1/(\alpha - 1)} N(r) = (\alpha \Gamma(\alpha))^{1/(\alpha - 1)} \text{  a.s.}
\end{equation}
The strategy of the proof will be to bound the process $(N_n(r), r \geq 0)$ from above by the process $(N(r), r \geq 0)$ started at time $(1 - \delta)\alpha \Gamma(\alpha) n^{1 - \alpha}$, when there will typically be more than $n$ blocks, and from below by the process $(N(r), r \geq 0)$ started at time $(1 + \delta) \alpha \Gamma(\alpha) n^{1 - \alpha}$, when there will typically be fewer than $n$ blocks.

Choose $\delta > 0$ sufficiently small that
\begin{equation}\label{delta1}
(1 + \delta) \bigg( \frac{\alpha \Gamma(\alpha)}{(1 - \delta)\alpha \Gamma(\alpha) + r} \bigg)^{1/(\alpha - 1)} \leq (1 + \varepsilon) m(r)
\end{equation}
and
\begin{equation}\label{delta2}
(1 - \delta) \bigg( \frac{\alpha \Gamma(\alpha)}{(1 + \delta)\alpha \Gamma(\alpha) + r} \bigg)^{1/(\alpha - 1)} \geq (1 - \varepsilon) m(r).
\end{equation}
For $r > 0$, define the event $$G(r) = \big\{(1 - \delta) (\alpha \Gamma(\alpha))^{1/(\alpha - 1)} r^{-1/(\alpha - 1)}n \leq N(n^{1-\alpha} r) \leq (1 + \delta) (\alpha \Gamma(\alpha))^{1/(\alpha - 1)} r^{-1/(\alpha - 1)}n\big\}.$$  
For $n \in \N$, let $g(n) = (1 + \delta)(\alpha \Gamma(\alpha) + h(n))$.  Then $\lim_{n \rightarrow \infty} g(n)n^{1-\alpha} = 0$, so (\ref{bbseq}) gives
\begin{equation}\label{PG1}
\lim_{n \rightarrow \infty} P\big(G(r) \mbox{ occurs for all }r \in [0, g(n)] \big) = 1
\end{equation}
and
\begin{equation}\label{timeton}
\lim_{n \rightarrow \infty} P\big( N((1 + \delta) \alpha \Gamma(\alpha) n^{1-\alpha}) \leq n \leq N((1 - \delta) \alpha \Gamma(\alpha) n^{1-\alpha}) \big) = 1.
\end{equation}
Therefore,
\begin{align*}
&P \big( N_n(n^{1-\alpha} r) > (1 + \varepsilon) m(r) n \mbox{ for some }r \in [0, h(n)]\big) \nonumber \\
&\hspace{.5in} \leq P \big( N((1 - \delta) \alpha \Gamma(\alpha) n^{1-\alpha}) < n \big) \nonumber \\
&\hspace{.8in}+ P \big( N(((1 - \delta)\alpha \Gamma(\alpha)+ r) n^{1 - \alpha}) > (1 + \varepsilon) m(r) n \mbox{ for some }r \in [0, h(n)] \big),
\end{align*}
which tends to zero as $n \rightarrow \infty$ by (\ref{delta1}), (\ref{PG1}), and (\ref{timeton}).  Likewise,
\begin{align*}
&P \big( N_n(n^{1-\alpha} r) < (1 - \varepsilon) m(r) n \mbox{ for some }r \in [0, h(n)]\big) \nonumber \\
&\hspace{.5in} \leq P \big( N((1 + \delta) \alpha \Gamma(\alpha) n^{1-\alpha}) > n \big) \nonumber \\
&\hspace{.8in}+ P \big( N(((1 + \delta)\alpha \Gamma(\alpha) + r) n^{1 - \alpha}) < (1 - \varepsilon) m(r) n \mbox{ for some }r \in [0, h(n)] \big),
\end{align*}
which tends to zero as $n \rightarrow \infty$ by (\ref{delta2}), (\ref{PG1}), and (\ref{timeton}).  The result follows.
\end{proof}

\subsection{Functionals of the block counting process}

We again use the notation of section \ref{resultsec}, so that $n=X_0 > X_1 > \cdots > X_{\tau_n}=1$ is the block-counting process of the beta $n$-coalescent and $Y_k = X_k - X_{k+1}$.  

\begin{Prop}\label{quadrature}
Let $f:(0,1] \to \mathbb R$ be a positive, decreasing, differentiable function such that
\[ |f'(x)| \le c x^{-\zeta-1} \]
for some $c>0$. If $\zeta < 1$, then $f$ is integrable and 
\[ \gamma \sum_{k<\tau_n} f\Big(\frac {X_k}n\Big) = n \int_0^1 f(x)\, dx + o_P(n) . \]
If moreover $\zeta< 1/\alpha$, i.e. $f \in\mathscr F$, then
\[\gamma \sum_{k<\tau_n} f\Big(\frac {X_k}n\Big) = n \int_0^1 f(x)\, dx -  \sum_{k<\tau_n} f\Big(\frac {X_k}n\Big) (Y_k-\gamma) + o_P(n^{1/\alpha}).\]
\end{Prop}

\begin{proof} 
Choose $\varepsilon >0$.
Let $0=a_0< a_1 < \cdots < a_m=1$ be a partition of $[0,1]$. Then by the assumptions on $f$
\[0\le  \sum _{  \tau_n(a_1)\le k<\tau_n } f\bigg(\frac {X_k}n\bigg) \le \sum_{1 \le j \le a_1n} f\bigg(\frac jn\bigg) \le n \int_0^{a_1} f(x)\: dx \le \tfrac \varepsilon 3 n , \]
if only $a_1$ is sufficiently small.

Also for $i >1$ by monotonicity and Lemma \ref{Xstop}
\[   \sum_{\tau_n(a_{i}) \le k < \tau_n(a_{i-1})} f\bigg(\frac {X_k}n\bigg) \ge f(a_i) (\tau_n(a_{i-1})-\tau_n(a_{i})) = f(a_i)\frac{(a_{i}-a_{i-1})n}\gamma + o_P(n) \]
which implies that
\[ \sum_{k<\tau_n(a_1) } f\bigg(\frac {X_k}n\bigg) \ge \frac n\gamma \int_{a_1}^{1} f(x)\: dx - \tfrac \varepsilon3 n + o_P(n) ,\]
if only the partition is chosen fine enough. Combining the estimates we obtain
\[ P \bigg(\sum_{k < \tau_n}f\bigg(\frac {X_k}n\bigg) < \frac n\gamma  \int_0^1 f(x)\: dx - \varepsilon n\bigg) =o(1).\]
In the same manner we may bound $\sum_{k< \tau_n} f(X_k/n)$ from above such that the first claim follows.

%

As to the second one, from a Taylor expansion with $X_{k+1} \le \bar X_k \le X_k$,
\[ \int_{X_{k+1}/n}^{X_{k}/n} f(x)\, dx = f\Big(\frac {X_k}n\Big) \frac {Y_k}n  + \frac 12 f'\Big(\frac {\bar X_k}n\Big) \Big(\frac{Y_k} n\Big)^2 . \]
Therefore $R_n$, given by
\[ \gamma \sum_{k<\tau_n } f\Big(\frac {X_k}n\Big) = n \int_{1/n}^1 f(x)\, dx -  \sum_{k<\tau_n} f\Big(\frac {X_k}n\Big) (Y_k-\gamma) + R_n,\]
fulfils by assumption
\[ |R_n| \le \sum_{k=0}^{\tau_n-1} \Big|f'\Big(\frac {\bar X_{k}}n\Big)\Big| \frac {Y_k^2}n\le c \sum_{k=0}^{\tau_n-1} \Big( \frac{X_{k+1}}n\Big)^{-\zeta-1}\frac{Y_{k}^2}n \]
or
\[ |R_n| \le cn^\zeta \Big(\max_{k< \tau_n}\frac{X_k}{X_{k+1}}\Big)^{\zeta+1} \sum_{k=0}^{\tau_n-1} X_k^{-\zeta-1}Y_k^2.\]
Because of Lemma \ref{UgivenX}, equation \eqref{Stir}, and the fact that $\alpha < 2$, for $k<\tau_n$ we get
\[ E( Y_k^2 \mid X_k) \le c \sum_{i < X_k} i^2 q_i \le c X_k^{2-\alpha} \]
and therefore
\begin{align*} E\Big( \sum_{k < \tau_n} X_{k}^{-\zeta-1}Y_{k}^2\Big)  \le c E\Big(  \sum_{ k < \tau_n } X_{k}^{1-\alpha-\zeta}\Big)\le c \sum_{j \ge 1} j^{1-\alpha-\zeta}.
\end{align*}
Furthermore $1-\alpha-\zeta = \frac 1\alpha - \zeta - \frac 1\alpha(\alpha-1)^2-1 <-1$, if only $\zeta$ is chosen sufficiently close to $1/\alpha$, such that the right-hand series is convergent.
Altogether in view of Lemma \ref{Xratio} we obtain 
\[R_n= O_P(n^\zeta)=o_P(n^{1/\alpha}).\] 
Finally, since $\zeta < 1/\alpha < 1â$ it follows from (\ref{fbound}) that
\[ n\int_0^{1/n} f(x)\, dx \le c n\Big(\frac{1}n \Big)^{1-\zeta}= c n^{\zeta} =o(n^{1/\alpha}). \]
This gives the second assertion. 
\end{proof}


\begin{Lemma}\label{bigUsmallU}
Let $f\in \mathscr F$. Then for any $\eta >0$ there is an $\varepsilon >0$ such that for all $n\ge 1$
\[ P \Big( \Big|  d + n^{-1/\alpha} \sum_{k< \tau_n}f\Big(\frac {X_{k}} n \Big) (Y_k 1_{\{  f(X_k/n) Y_k \le \varepsilon n^{1/\alpha} \}} - \gamma)  \Big| >\eta\Big) \le \eta \]
with
\begin{align}\label{constd} d= d(\varepsilon)= \varepsilon^{1-\alpha}\frac{\alpha}{\Gamma(2 - \alpha)} \int_0^1 f(x)^{\alpha} \: dx. 
\end{align}
\end{Lemma}


\begin{proof}
Suppressing the dependence on $n$ in the notation, let
\begin{align*} A_k&= \{  f(X_k/n) Y_k \le \varepsilon n^{1/\alpha} \},\\
\gamma(j) &=  E( Y_01_{A_0}   \mid X_0=j).
\end{align*}
Since $(X_k)$ is a Markov chain and $\tau_n$ is a stopping time, the random variables
\[  f \bigg( \frac{X_k}{n} \bigg) \big(Y_k 1_{A_k} - \gamma(X_{k})\big) 1_{\{k < \tau_n \}} \]
have zero mean and are uncorrelated.  Therefore,
\[ \Var \bigg( \sum_{k< \tau_n} f\Big(\frac {X_{k}} n \Big)(Y_k 1_{A_k} - \gamma(X_{k})) \bigg) \le \sum_{k=0}^{n-1} E\bigg( f\Big(\frac {X_{k}} n \Big)^2 (Y_k 1_{A_k} - \gamma(X_k))^2 1_{\{k < \tau_n \}} \bigg) . \]
From Lemma \ref{UgivenX} and \eqref{Stir}, we see that 
$$E((Y_0 1_{A_0} - \gamma(X_0))^2 \mid X_0 = j) \leq E(Y_0^2 1_{A_0} \mid X_0 = j) \leq c \sum_{i \leq \varepsilon n^{1/\alpha}/f(j/n)} i^2 q_i \leq c \bigg( \frac{\varepsilon n^{1/\alpha}}{f(j/n)} \bigg)^{2-\alpha}.$$ Thus,
\[ \Var \bigg( \sum_{k< \tau_n} f\Big(\frac {X_{k}} n \Big)(Y_k 1_{A_k} - \gamma(X_{k})) \bigg) \le c \varepsilon^{2-\alpha} n^{2/\alpha - 1} \sum_{k=0}^{n-1} E \bigg( f \bigg( \frac{X_k}{n} \bigg)^{\alpha} \bigg). \]
Because $\zeta < 1/\alpha$, we have, using (\ref{fbound}),
\begin{align*}\sum_{k=0}^{n-1} E \bigg( f \bigg( \frac{X_k}{n} \bigg)^{\alpha} \bigg) \leq \sum_{j=1}^n f\bigg( \frac{j}{n} \bigg)^{\alpha} \leq c \sum_{j=1}^n \bigg( \frac{j}{n} \bigg)^{-\alpha \zeta} \leq cn,
\end{align*} and it follows that
\[ \Var \bigg( n^{-1/\alpha} \sum_{k< \tau_n} f\Big(\frac {X_{k}} n \Big)(Y_k 1_{A_k} - \gamma(X_{k})) \bigg) \le c \varepsilon^{2-\alpha}. \]  Thus, by Chebyshev's Inequality, if $\varepsilon$ is sufficiently small, then
\begin{equation}\label{gammaXk}
P\bigg( \bigg|  n^{-1/\alpha} \sum_{k< \tau_n} f\Big(\frac {X_{k}} n \Big)(Y_k 1_{A_k} - \gamma(X_{k}) ) \bigg| >\eta\bigg) \le \eta.
\end{equation}

It remains to replace $\gamma(X_k)$ by $\gamma$ in this formula.  From
\[ \gamma(j) = \sum_{i \le \varepsilon n^{1/\alpha}/f(j/n)} i P(Y_0=i\mid X_0=j) \]
we get from Lemma \ref{UgivenX} and \eqref{Stir} the estimate, uniform in $n$ and $k$,
\[ \bigg| \gamma(j) - \sum_{i \le (j - 1) \wedge \varepsilon n^{1/\alpha}/f(j/n)} i q_i \bigg| \le \frac cj \sum_{i\le j-1} i^2 q_i = O(j^{1-\alpha}).\]
From \eqref{Stir} and \eqref{pkmean},
\[ \sum_{i \le (j - 1) \wedge \varepsilon n^{1/\alpha}/f(j/n)} i q_i = { \gamma - } \frac{\alpha}{(\alpha -1)\Gamma(2-\alpha)} \bigg( \frac{\varepsilon n^{1/\alpha}}{f(j/n)} \bigg)^{1-\alpha}(1+o(1)) + O(j^{1-\alpha}),  \]
where the $o(1)$ goes to 0 with $n$ going to infinity, uniformly in $j$. Putting these formulas together we arrive at
\[ \gamma(j) = \gamma - \frac{\alpha}{(\alpha -1)\Gamma(2-\alpha)} \varepsilon^{1-\alpha} n^{-1+1/\alpha} f(j/n)^{\alpha - 1}
(1+o(1)) + O(j^{1-\alpha}) . \]
It follows that
\begin{align} \label{3terms} \sum_{k<\tau_n} f\Big(\frac{X_k}n\Big) \gamma(X_{k}) = \gamma&\sum_{k<\tau_n} f\Big(\frac{X_k}n\Big) -  \frac{\alpha+o(1)}{(\alpha -1)\Gamma(2-\alpha)} \varepsilon^{1-\alpha} n^{-1 + 1/\alpha}\sum_{k<\tau_n}f\Big(\frac{X_k}n\Big)^{\alpha} \nonumber \\ &\mbox{}+  O\bigg(\sum_{k<\tau_n} f \Big(\frac{X_k}n\Big)X_{k}^{1-\alpha}\bigg).
\end{align}
Now, since $1-\alpha-\zeta = \frac 1\alpha - \zeta - \frac 1\alpha(\alpha-1)^2-1 <-1$, if $\zeta$ is sufficiently close to $1/\alpha$,
\begin{align*}
\sum_{k<\tau_n} f\Big(\frac{X_k}n\Big)X_{k}^{1-\alpha} \le c n^\zeta \sum_{k < \tau_n} X_k^{1-\alpha-\zeta} \le cn^\zeta\sum_{j\ge 1} j^{1-\alpha-\zeta} = O(n^\zeta)=o(n^{1/\alpha}).
\end{align*}
Because ${|}\frac{d}{ds} f(s)^{\alpha}{|} = \alpha f(s)^{\alpha - 1}{|}f'(s) {|}\leq c s^{-\zeta(\alpha - 1)} s^{-\zeta - 1} = cs^{-\alpha \zeta - 1}$ and $\alpha \zeta < 1$, we may estimate the middle term in the right-hand side of (\ref{3terms}) by applying the first statement of Proposition \ref{quadrature}.  This implies
\[ \sum_{k<\tau_n} f\Big(\frac{X_k}n\Big) \gamma(X_{k}) = \gamma\sum_{k<\tau_n} f\Big(\frac{X_k}n\Big) -  \frac{\alpha}{\Gamma(2-\alpha)} \varepsilon^{1-\alpha} n^{1/\alpha} \int_0^1 f(s)^{\alpha} \, ds + o_P(n^{1/\alpha}), \]
which, combined with (\ref{gammaXk}), implies the result.
\end{proof}

\subsection{Proof of Theorem \ref{onetime}}

Let $f\in \mathscr F$. We also assume that $f(x) $ is decreasing, $x^2f(x)$ is increasing and $f(x)\ge 1$ for all $x$; see the remark at the beginning of Section \ref{proofsec}.

 Let
\begin{equation}\label{Zdef}
Z_n= \int_{(0,1]} f(x) \: M_n(dx).
\end{equation}
We have to show  that $$ \gamma \sum_{k < \tau_n} f \bigg( \frac{X_k}{n} \bigg)= n \int_0^1 f(s) \: ds  - n^{1/\alpha}Z_n + o_P(n^{1/\alpha}).$$  In view of Proposition \ref{quadrature} it suffices to show that 
\begin{equation}\label{maincoupling}
\sum_{k < \tau_n} f \bigg( \frac{X_k}{n} \bigg) (Y_k - \gamma) =n^{1/\alpha} Z_n + o_P(n^{1/\alpha}).
\end{equation}

Enumerate the points of $\Theta_n$ as $(x_i, y_i)_{i=1}^{\infty}$ and for $\varepsilon >0$ let 
\begin{align} 
S_n(\varepsilon) = \sum_{f(x_i)y_i > \varepsilon} f(x_i)y_i.
\label{finitesum}
\end{align}
First let us check that 
\[ S_n(\varepsilon)- E(S_n(\varepsilon))\to Z_n \]
in probability as $\varepsilon \to 0$. For this purpose note that from \eqref{intensity}
\begin{equation}\label{ESn}
E(S_n(\varepsilon))= \int_0^1\int_{\varepsilon/f(x)}^\infty f(x)y\: \nu(dx,dy)= \varepsilon^{1-\alpha}\frac{\alpha}{\Gamma(2-\alpha)}\int_0^1 f(x)^\alpha\: dx,
\end{equation}
which is finite by our assumptions on $f$. Thus the sum in \eqref{finitesum} has a.s. finitely many summands. Also for $\eta >0$
\[ \text{Var} \bigg( \sum_{i\ge 1} f(x_i)y_i \cdot 1_{\{\eta \le y_i,\, f(x_i)y_i \le \varepsilon\}} \Bigg) = \int_0^1 \int_\eta^{\varepsilon/f(x)} f(x)^2y^2\cdot 1_{\{\eta \le \varepsilon/f(x)\}} \: \nu(dx,dy) .\]
Letting $\eta \to 0$ we obtain
\[ \text{Var}(Z_n-S_n(\varepsilon))= \int_0^1 \int_0^{\varepsilon/f(x)} f(x)^2 y^2 \: \nu(dx,dy) = \varepsilon^{2-\alpha}\frac{\alpha(\alpha-1)}{\Gamma(3-\alpha)} \int_0^1 f(x)^\alpha\: dx, \]
which goes to 0 as $\varepsilon \to 0$.


It now follows from Chebyshev's Inequality and the fact that $E[Z] = 0$ that for all $\eta >0$ and  sufficiently small $\varepsilon$, $$P (|Z_n - (S_n(\varepsilon)-E(S_n(\varepsilon)) | > \eta ) < \eta.$$  Because $E(S_n(\varepsilon))$ equals the constant $d$ in \eqref{constd}, it follows from Lemma \ref{bigUsmallU} that for sufficiently small $\varepsilon$ we have $$P \bigg( \bigg| n^{-1/\alpha} \sum_{k < \tau_n} f \bigg( \frac{X_k}{n} \bigg) (Y_k 1_{\{ f(X_k/n) Y_k \leq \varepsilon n^{1/\alpha}\}} - \gamma) - (Z_n- S_n(\varepsilon)) \bigg| > 2 \eta \bigg) \leq 2 \eta$$ for sufficiently large $n$.  Thus, to show (\ref{maincoupling}), it suffices to show that for all $\varepsilon > 0$ we have
\begin{equation}\label{couplejumps}
P \bigg( \bigg| n^{-1/\alpha} \sum_{k < \tau_n} f \bigg( \frac{X_k}{n} \bigg) Y_k 1_{\{ f(X_k/n)Y_k > \varepsilon n^{1/\alpha}\}} - S_n(\varepsilon) \bigg| > \eta \bigg) < \eta
\end{equation}
for sufficiently large $n$.  For this, we will use the following lemmas.  Again let $\Theta_n$ be the Poisson point process constructed in section \ref{twoPPP} from another Poisson point process $\Psi_n$.  Denote the points of $\Theta_n$ by $(x_i, y_i)_{i=1}^{\infty}$ with $x_i = m(r_i)$ and $y_i = m(r_i) u_i$, where $(s_i, u_i)$ are the points of $\Psi_n$ and $r_i = -s_i$.

\begin{Lemma}\label{couplelem}
Let $\delta > 0$, $\varepsilon > 0$.  With probability tending to 1 as $n \rightarrow \infty$, for all $i $ such that $f(x_i)y_i >\varepsilon$, there exists a positive integer $k_i$ such that the following hold:
\begin{enumerate}
\item[(i)] $R_{k_i}=n^{1-\alpha} r_i$, i.e. at time $n^{1-\alpha}r_i$ there is a merger in the coalescent back from time 0.
\item[(ii)]  The block size $X_{k_i}= N_n(n^{1-\alpha} r_i-)$ and the  merger's size $Y_{k_i}=N_n(n^{1-\alpha} r_i-) -N_n(n^{1-\alpha} r_i)$ fulfil $$(1 - \delta) x_i  \leq \frac{X_{k_i}}{n} \leq (1 + \delta) x_i, \hspace{.3in} (1 - \delta)y_i  \leq \frac{Y_{k_i}}{n^{1/\alpha}} \leq (1 + \delta) y_i .$$
\end{enumerate}
\end{Lemma}

\begin{proof} The points $(n^{1-\alpha} s_i, n^{-1+1/\alpha}u_i)$ are points of $\Upsilon'$.  Consider those points, which in addition fulfil  $f(x_i)y_i> \varepsilon$.  First we verify that the probability that some of these points do not belong to $\Upsilon_n$ is asymptotically vanishing. 
The expected number of indices $i$ with $f(x_i)y_i > \varepsilon$, that is $u_i> \varepsilon/(f(m( r_i ))m( r_i ))$, such that $(n^{1-\alpha} s_i, n^{-1 + 1/\alpha}u_i)$ is not a point of $\Upsilon_n$ is $$\int_0^{\infty} \int_{\varepsilon/(f(m( r ))m( r ))}^{\infty} \frac{1}{\Gamma(\alpha) \Gamma(2 - \alpha)} u^{-1-\alpha} (1 - q(n^{-1+1/\alpha} u)) \: du \: dr.$$  Note that $1 - q(u) \leq cu$ for all $u\geq 0$.  Using this bound when $r \leq n^{\gamma}$ and the bound $1 - q(u) \leq 1$ when $r \geq n^{\gamma}$, we see that the above expectation is at most
\begin{align*}
&c n^{-1 + 1/\alpha} \int_0^{n^{\gamma}} \int_{\varepsilon/(f(m( r))m( r))}^{\infty} u^{-\alpha} \: du \: dr + c \int_{n^{\gamma}}^{\infty} \int_{\varepsilon/(f(m( r))m( r) )}^{\infty} u^{-1-\alpha} \: du \: dr\nonumber \\
&\hspace{.5in} = c n^{-1 + 1/\alpha} \int_0^{n^{\gamma}} \bigg( \frac{ f(m(r))m(r)}{\varepsilon} \bigg)^{\alpha - 1} \: dr + c \int_{n^{\gamma}}^{\infty} \bigg( \frac{f(m( r))m( r) }{\varepsilon} \bigg)^{\alpha} \: dr. \nonumber
\end{align*}
Making the substitution $x = m( r)$ and using (\ref{fbound}) again, we get that this expression is bounded above by $$c n^{-1+1/\alpha} \int_{m(n^{\gamma})}^1 x^{-1} f(x)^{\alpha - 1} \: dx + c \int_0^{m(n^{\gamma})} f(x)^{\alpha} \: dx \leq c n^{-1 + 1/\alpha} m(n^{\gamma})^{-\zeta(\alpha - 1)} + c \int_0^{m(n^{\gamma})} f(x)^{\alpha} \: dx,$$
which tends to zero as $n \rightarrow \infty$ if $\gamma > 0$ is sufficiently small.  By Markov's Inequality, the probability that $(n^{1-\alpha} s_i, n^{-1 + 1/\alpha}u_i)$ is a point of $\Upsilon_n$ for all $i$ with $f(x_i)y_i>\varepsilon$ tends to 1 as $n \rightarrow \infty$.

Second we show that the probability that no more than one of the remaining $N_n(n^{1-\alpha}r_i-)$ blocks takes part in a merging event at time $n^{1-\alpha}r_i$ is asymptotically vanishing.
We choose a function $h: \N \rightarrow \infty$ such that $\lim_{n \rightarrow \infty} h(n) = \infty$, and $h(n) = o(n^{\alpha-1})$.  The expected number of indices $i$ such that $r_i > h(n)$ and $f(x_i)y_i > \varepsilon$ is at most 
\begin{align*}
\int_{h(n)}^{\infty} \int_{\varepsilon/(f(m( r ))m( r ))}^{\infty} &\frac{1}{\Gamma(\alpha) \Gamma(2 - \alpha)} u^{-1-\alpha}  \: du \: dr\\&=c\int_{h(n)}^{\infty} \bigg( \frac{f(m( r))m( r)}{\varepsilon} \bigg)^{\alpha} \: dr \le c \int_0^{m(h(n))} f(x)^{\alpha} \: dx, 
\end{align*} 
which tends to zero as $n \rightarrow \infty$.  Thus, we may assume $r_i \leq h(n)$ for all $i$ with $f(x_i)y_i> \varepsilon$.  In particular, because of 
$y_i > \varepsilon/f(x_i) \ge cx_i^\zeta = c m(r_i)^\zeta \ge c r_i^{- \zeta/(\alpha-1)} \ge c h(n)^{-\zeta/(\alpha-1)}$, this implies by our assumptions on $\zeta$ and $h(n)$ that
\begin{align}\label{yn}
n^{1/\alpha} y_i \ge cn^{1/\alpha}  h(n)^{-\zeta/(\alpha-1)}\to \infty.
\end{align}
Also, because $E(S_n(\varepsilon))$ does not depend on $n$ and is finite, it suffices to show that points $(x_i,y_i)$ with $f(x_i)y_i>\varepsilon$ and $r_i\le h(n)$ lead to mergers fulfilling condition  (ii) with high probability for large $n$.

At time $n^{1-\alpha} r_i$, the number of blocks of the beta coalescent is reduced by $(A_i - 1) \vee 0$, where $A_i$ has a binomial distribution with parameters $n_i = N_n(n^{1-\alpha}r_i-)$ and $p_i = n^{-1+1/\alpha} u_i$.  By Chebyshev's Inequality, we have
\begin{equation}\label{Chebbinom}
P(|A_i- n_i p_i| > \theta n^{1/\alpha} y_i\mid  n_i, p_i, y_i) \leq \frac{n_i p_i(1 - p_i)}{(\theta n^{1/\alpha} y_i)^2} \leq \frac{n_i p_i}{(\theta n^{1/\alpha} y_i)^2}.
\end{equation}
Let $\theta = \delta/3$.  Lemma \ref{Ntlem} implies that with probability tending to 1 as $n \rightarrow \infty$, we have
\begin{equation}\label{Nnbd}
(1 - \theta) m(r_i) n \leq N_n(n^{1-\alpha} r_i-) \leq (1 + \theta) m(r_i) n,
\end{equation}
and on this event the right-hand side of (\ref{Chebbinom}) is in view of \eqref{yn} bounded above by 
\begin{align*} 
\frac{(1 + \theta) m(r_i) n^{1/\alpha} u_i}{(\theta n^{1/\alpha} y_i)^2} = \frac{1 + \theta}{\theta^2 n^{1/\alpha} y_i} { \le c\frac{(1 + \theta) }{\theta^2 n^{1/\alpha}} h(n)^{\frac{\zeta}{\alpha-1}}=o(n^{\zeta-1/\alpha})}.
\end{align*}
Due to our assumption on $\zeta$ the right-hand side  tends to zero as $n \rightarrow \infty$. Taking expectations in (\ref{Chebbinom}) gives $$ P(|A_i - n_i p_i| > \theta n^{1/\alpha} y_i) =o(1).$$  Combining this bound with (\ref{Nnbd}) gives $$\lim_{n \rightarrow \infty} P(|A_i - n^{1/\alpha} y_i| > 2 \theta n^{1/\alpha} y_i) = 0.$$  Because of \eqref{yn} we get
\begin{equation}\label{finalAi}
\lim_{n \rightarrow \infty} P(|(A_i - 1) - n^{1/\alpha} y_i| > 3 \theta n^{1/\alpha} y_i) = 0.
\end{equation}
Thus, with probability tending to one as $n \rightarrow \infty$, the beta coalescent must have a merger at time $n^{1-\alpha} r_i$, which means $n^{1-\alpha} r_i = R_{k_i}$ for some positive integer $k_i$.  That is,  condition (i) in the statement of the lemma holds. Now, because $\delta = 3 \theta$, condition (ii) is a consequence of Lemma \ref{Ntlem} and (\ref{finalAi}).  
\end{proof}

\begin{Lemma}\label{couplelem2}
Let $\delta > 0$, $\varepsilon >0$. With probability going to 1 as $n \to \infty$, for all $0\le k < \tau_n$ with $f(\frac{X_k}n)\frac{Y_k}{n^{1/\alpha}} > \varepsilon$, 
\begin{align}(1-\delta)x_{i_k}  \le \frac {X_k}n\le (1+\delta)x_{i_k},\hspace{.3in} (1-\delta)y_{i_k}  \le \frac{Y_k}{n^{1/\alpha}}\le (1+\delta)y_{i_k} ,
\label{estimate2}
\end{align}
where $i_k$ is determined by $X_{k}= N_n(n^{1-\alpha} r_{i_k}-)$ or $Y_{k}=N_n(n^{1-\alpha} r_{i_k}-) -N_n(n^{1-\alpha} r_{i_k})$.
\end{Lemma}

\begin{proof} First we estimate the expectation of 
\[ S_n(\varepsilon,\gamma)= n^{-1/\alpha}\sum_{k<\tau_n} f\bigg( \frac {X_k}n\bigg) Y_k 1_{\{f(X_k/n)Y_k > \varepsilon n^{1/\alpha}\}} 1_{\{X_k \le \gamma n\}}\]
with $0<\gamma \le 1$. From Lemma \ref{UgivenX} and \eqref{Stir}
\[ E(Y_0 1_{\{f(X_0/n)Y_0 > \varepsilon n^{1/\alpha}\}} \mid X_0=j) \le c\sum_{i>\varepsilon n^{1/\alpha}/f(j/n)} iq_i \le c \bigg(\frac {\varepsilon n^{1/\alpha}}{f(j/n)}\bigg)^{1-\alpha}.\]
Consequently
\[  E\bigg( f\bigg( \frac {X_k}n\bigg)Y_k 1_{\{f(X_k/n)Y_k > \varepsilon n^{1/\alpha}\}}1_{\{X_k \le \gamma n\}}\bigg) \le c \varepsilon^{1-\alpha} n^{1/\alpha-1} E\bigg(f\bigg( \frac {X_k}n\bigg)^\alpha1_{\{X_k \le \gamma n\}}\bigg)\]
and, since $f$ is assumed to be decreasing,
\begin{align*} E(S_n(\varepsilon,\gamma) ) &\le c \varepsilon^{1-\alpha} n^{-1} \sum_{k=0}^n E\bigg(f\bigg( \frac {X_k}n\bigg)^\alpha1_{\{X_k \le \gamma n\}}\bigg) \\ &\le c \varepsilon^{1-\alpha} n^{-1} \sum_{1\le j\le \gamma n}f\bigg(\frac jn\bigg)^\alpha\le  c \varepsilon^{1-\alpha}\int_0^\gamma f(s)^\alpha\: ds.
\end{align*}
By the assumptions on $f$, the integral is finite. 

For $\gamma =1$ we see that $E(S_n(\varepsilon,\gamma))$ is uniformly bounded in $n$.
Since each positive summand of $S_n(\varepsilon,\gamma)$ is bigger than~$\varepsilon$, it follows that the number of positive summands is stochastically bounded. Therefore it is sufficient to verify that for any $0 \le k<\tau_n$ with $f(X_k/n)Y_k>\varepsilon n^{1/\alpha}$ the two formulas \eqref{estimate2}  hold with probability going to 1, uniformly in $k$.

Let $\theta >0$. Then there is a $\gamma >0$ such that $E(S_n(\varepsilon,\gamma)) \le \theta\varepsilon$.  Therefore, with probability at least $1 - \theta$, we have $X_k \geq \gamma n$ for all $k$ such that $f(X_k/n)Y_k>\varepsilon n^{1/\alpha}$.  In view of Lemma \ref{Ntlem}, since $N_n( r )$ is decreasing, this implies that with probability going to 1 we have $(1+\delta)m(r_{i_k})\ge \gamma$, which implies that $r_{i_k} \le T$ for some fixed constant $T<\infty$  and also
\[ (1-\delta)x_{i_k}  \le \frac {X_k}n\le (1+\delta)x_{i_k} .\]

Furthermore $X_k\ge \gamma n$ and $f(X_k/n)Y_k>\varepsilon n^{1/\alpha}$ imply
\[ Y_k > cn^{1/\alpha}\]
with $c=  \varepsilon/\sup_{x \ge \gamma} f(x) >0$. Let  $n_k= N_n(n^{1-\alpha}r_{i_k}-)$.
Since $r_{i_k} \le T$, by Lemma \ref{Ntlem} with probability going to 1
\begin{align} |n_kp_{i_k} -n^{1/\alpha} y_{i_k}|= |n^{-1}N_n(n^{1-\alpha}r_{i_k}-)- m(r_{i_k})|n^{1/\alpha} u_{i_k} \le \tfrac \delta 3 n^{1/\alpha}y_{i_k}.  
\label{estimate3}
\end{align}
Thus for $\eta=1/(2\alpha)$ and $n$ sufficiently large
\begin{align*}
 P(Y_k > cn^{1/\alpha}&, |Y_k -n^{1/\alpha} y_{i_k}| >  \delta y_{i_k} n^{1/\alpha} \mid n_k,p_{i_k},y_{i_k}) \\&\le P( Y_k>0, |Y_k+1 -n_kp_{i_k}| >  \tfrac \delta 2y_{i_k} n^{1/\alpha} \mid n_k,p_{i_k},y_{i_k})1_{\{y_{i_k} > n^{-\eta}\}} \\ &\mbox{} \qquad +P(Y_k > cn^{1/\alpha } \mid n_k,p_{i_k},y_{i_k}) 1_{\{y_{i_k} \le n^{-\eta}\}}.
 \end{align*}
Since $Y_k=(A_k-1)\vee 0$, where $A_k$ is binomial with parameters $n_k$ and $p_{i_k}$, it follows that
\begin{align*}
 P(Y_k > cn^{1/\alpha}&, |Y_k -n^{1/\alpha} y_{i_k}| >  \delta y_{i_k} n^{1/\alpha} \mid n_k,p_{i_k},y_{i_k}) \\&\le
 P(Y_k>0,|A_k-n_kp_{i_k}|>  \tfrac \delta 2 y_{i_k} n^{1/\alpha} \mid n_k,p_{i_k},y_{i_k}) 1_{\{y_{i_k} > n^{-\eta}\}} \\ &\mbox{} \qquad +P(A_k > cn^{1/\alpha } \mid n_k,p_{i_k},y_{i_k}) 1_{\{y_{i_k} \le n^{-\eta}\}}.
\end{align*}
Chebyshev's and Markov's inequality imply together with \eqref{estimate3}
\begin{align*}
 P(Y_k > cn^{1/\alpha}&, |Y_k -n^{1/\alpha} y_{i_k}| >  \delta y_{i_k} n^{1/\alpha} \mid n_k,p_{i_k},y_{i_k}) \\&\le
\frac 4{\delta^2} \frac{ n_kp_{i_k}(1-p_{i_k})}{(y_{i_k} n^{1/\alpha})^2}1_{\{y_{i_k} > n^{-\eta}\}} + \frac{n_kp_{i_k}}{cn^{1/\alpha}}1_{\{y_{i_k} \le n^{-\eta}\}}\\
&\le
\frac 8{\delta^2} \frac 1{y_{i_k} n^{1/\alpha}}1_{\{y_{i_k} > n^{-\eta}\}} + \frac {2y_{i_k}} c 1_{\{y_{i_k} \le n^{-\eta}\}} \le c n^{-1/(2\alpha)}
\end{align*}
and therefore
\[  P(Y_k > cn^{1/\alpha}, |Y_k -n^{1/\alpha} y_{i_k}| >  \delta y_{i_k} n^{1/\alpha}) \to 0. \]
Since $\theta$ was arbitrary, this gives our assertions.
\end{proof}

\begin{proof}[Proof of Theorem \ref{onetime}]
We make use of the following preliminary estimate for $f$.  Let $\xi >0$. Also let $0<\delta < 1$, $u,U\in (0,1]$, and $v,V\ge 0$. If $U\ge (1-\delta)u$ and $V\le (1+\delta)v$, then, since $f(x)$ is decreasing and $x^2f(x)$ is increasing,
\[ f(U)V \le f((1-\delta)u)(1+\delta)v \le (1-\delta)^{-2}(1+\delta)f(u)v \le (1+\xi)f(u)v, \]
if $\delta$ is sufficiently small (given $\xi$). Together with an analogous, reversed estimate we get that if $\xi >0$, and $(1-\delta)u \le U \le (1+\delta)u$, and $(1-\delta)v \le V \le (1+\delta)v$, then
\[ (1-\xi )f(u)v \le f(U)V\le (1+\xi)f(u)v ,\]
if $\delta>0$ is sufficiently small, which entails
\[ |f(U)V - f(u)v| \le \xi f(u)v.\]

Now let $\varepsilon, \xi > 0$. Then by Lemmas \ref{couplelem} and \ref{couplelem2}  it follows that, if $ f(x_i)y_i \notin [\varepsilon-\xi, \varepsilon + \xi]$ for all $i \ge 1$, then with probability going to 1
\[ f(x_i)y_i > \varepsilon \quad \Rightarrow \quad f\bigg(\frac {X_{k_i}}n\bigg)\frac {Y_{k_i}}{n^{1/\alpha}}> \varepsilon \qquad \text{and} \qquad  f\bigg(\frac {X_k}n\bigg)\frac {Y_k}{n^{1/\alpha}}> \varepsilon\quad \Rightarrow \quad f(x_{i_k})y_{i_k} > \varepsilon. \]
Thus, with $S_n(\varepsilon)$ defined as in (\ref{finitesum}),
\begin{align*} \bigg|n^{-1/\alpha} \sum_{k < \tau_n} f \bigg( \frac{X_k}{n} \bigg) &Y_k 1_{\{ f(X_k/n)Y_k > \varepsilon n^{1/\alpha}\}} - S_n(\varepsilon)\bigg| \\ &= \bigg| \sum_{i \ge 1} \bigg(f \bigg( \frac{X_{k_i}}{n} \bigg) \frac{Y_{k_i}}{n^{1/\alpha}} -f(x_i)y_i \bigg)1_{\{f(x_i)y_i> \varepsilon\}}  \bigg| \\
&\le \xi \sum_{i \ge 1} f(x_i)y_i 1_{\{f(x_i)y_i> \varepsilon\}} = \xi S_n(\varepsilon).
\end{align*}

We end up with
\begin{align*}
P \bigg( \bigg| n^{-1/\alpha} \sum_{k < \tau_n} &f \bigg( \frac{X_k}{n} \bigg) Y_k 1_{\{ f(X_k/n)Y_k > \varepsilon n^{1/\alpha}\}} - S_n(\varepsilon) \bigg| > \eta \bigg) \\
&\le P(\xi S_n(\varepsilon) > \eta) + P(\varepsilon-\xi \le f(x_i)y_i \le \varepsilon + \xi \text{ for some } i \ge 1) + o(1)\\
&\le \frac \xi \eta E(S_n(\varepsilon)) + P(\varepsilon-\xi \le f(x_i)y_i \le \varepsilon + \xi \text{ for some } i \ge 1) + o(1).
\end{align*}
Since the expectation is uniformly bounded in $n$ by (\ref{ESn}), the first term can be made arbitrarily small by decreasing $\xi$. The same is true for the second term, thus \eqref{couplejumps} follows, and the proof is finished.
\end{proof}

\end{document}